\documentclass[11pt]{article}

\usepackage{amsmath,amsthm,verbatim,amssymb,amsfonts,amscd, graphicx, tikz-cd}
\usepackage{graphics}
\usepackage{mathtools}
\theoremstyle{plain}
\newtheorem{theorem}{Theorem}

\newtheorem{lemma}{Lemma}
\newtheorem{remark}{Remark}
\newtheorem{proposition}{Proposition}

\theoremstyle{definition}

\newcommand\numberthis{\addtocounter{equation}{1}\tag{\theequation}}

\DeclareMathOperator{\im}{im}
\DeclareMathOperator{\KS}{KS}
\DeclareMathOperator{\Arf}{Arf}
\DeclareMathOperator{\lk}{lk}
\DeclareMathOperator{\SO}{SO}
\DeclareMathOperator{\GL}{GL}
\DeclareMathOperator{\interior}{int}

\DeclareMathOperator{\Spin}{Spin}
\DeclareMathOperator{\Pin}{Pin}

\DeclareMathOperator{\Hom}{Hom}

\begin{document}

\title{A relative version of Rochlin's theorem}
\author{Michael R. Klug}
\maketitle

\begin{abstract}
	Rochlin proved that a closed 4-dimensional connected smooth oriented manifold $X^4$ with vanishing second Stiefel-Whitney class has signature $\sigma(X)$ divisible by 16.  This was generalized by Kervaire and Milnor to the statement that if $\xi \in H_2(X;\mathbb{Z})$ is an integral lift of an element in $H_2(X; \mathbb{Z}/2\mathbb{Z})$ that is dual to $w_2(X)$, and if $\xi$ can be represented by an embedded sphere in $X$, then the self-intersection number $\xi^2$ is divisible by 16.  This was subsequently generalized further by Rochlin  and various alternative proofs of this result where given by Freedman, Kirby, and Matsumoto.  

	We give further generalizations of this result concerning 4-manifolds with boundary.  Given a smooth compact orientable four manifold $X^4$ with integral homology sphere boundary and a connected orientable characteristic surface with connected boundary $F^2$ properly embedded in $X$, we prove a theorem relating the Arf invariant of $\partial F$, and the Arf invariant of $F$, and the Rochlin invariant of $\partial X$.  We then proceed to generalize this result to the case where $X$ is a topological compact orientable 4-manifold (which brings in the Kirby-Siebenmann invariant), $\partial F$ is not connected (which brings in the condition of being proper as a link), $F$ is not orientable (which brings in Brown invariants), and finally where $\partial X$ is an arbitrary 3-manifold (which brings in pin structures).  The final result gives a ``combinatorial'' description of the Kirby-Siebenmann invariant of a compact orientable 4-manifold with nonempty boundary.  
\end{abstract}

\section{Introduction}

The Arf invariant of a quadratic form is an algebraic invariant of the form; by using topological constructions of such quadratic forms, there are several different topological contexts in which the Arf invariant makes an appearance.  For example, in dimension two, we have the Arf invariant of a spin structure on a surface; in dimension three, the simplest example is the Arf invariant of a knot in $S^3$; in dimension four, there is an Arf invariant of a characteristic surface in a 4-manifold.  

This paper considers the relationship between different appearances of the Arf invariant in low-dimensional topology as well as other related invariants, namely the Rochlin invariant of an integral homology sphere and the Kirby-Sienbenmann invariant of a 4-manifold.  Thus this paper amounts to an exercise in ``Invariantology'' (the study of the relationships between different invariants).  

  The most basic setup that we work with is as follows:  We are given a smooth compact 4-manifold $X$ whose boundary $\partial X$ is an integral homology 3-sphere.  Furthermore, $F^2$ is a orientable characteristic surface with connected nonempty boundary properly embedded in $X$ (see section \ref{sec:lemmas} for discussion of characteristic surfaces).  Note that we have the Arf invariant of the knot $\partial F$ in $\partial X$, the  Arf invariant of the characteristic surface $F$, the Rochlin invariant of the boundary three manifold $\partial X$, and the signature of the 4-manifold $X$ (see section \ref{sec:lemmas} for discussion of these invariants).   Then our main result is that we have the relationship

	$$
	\Arf(F) + \Arf(\partial F) = \frac{\sigma(X) - [F]^2}{8} + \mu(\partial X)  \pmod 2
	$$
  (see Theorem \ref{main_thm}). 

This relationship generalizes a theorem of Rochlin concerning the signature of closed 4-manifolds and the proof of our result involves a reduction to a certain version of this theorem.  In addition to Rochlin's work, other related results have been obtained by Robertello, Ac\~una-Gonz\'alez, Gordon, Yasuhara, and Kirby.  

If $F$ is a disk and $\partial X = S^3$, then we recover Robertello's original definition of the Arf invariant \cite{robertello}.  If instead we take $X$ to be an even 4-manifold and $F$ to be the undisk (i.e., a disk bounding an unknot in $\partial X$ pushed into $X$), then we recover the definition of the Rochlin invariant $\mu$.  When we take $F$ to be a disk, then we recover a theorem of Gordon in \cite{gordon1975knots}.  When we restrict to $\partial X = S^3$, then we obtain Theorem 2.2 of \cite{yasuhara1996connecting}.  

 Ac\~una-Gonz\'alez gave a fundamental relationship between the Arf invariant of a knot and the Rochlin invariant of a homology sphere by showing that if $K$ is a knot in $S^3$, then the Rochlin invariant of $+1$-surgery on $K$ is equal to the Arf invariant of $K$ (and similarly for $-1$-surgery) \cite{gonzalez1970dehn}.  This result follows as a special case of Theorem \ref{main_thm} as we indicate at the end of section \ref{sec:main}.  

With this result in hand, some natural questions arise.  How does our result extend to topological manifolds? The answer, as we shall see, is that we must simply add the Kirby-Siebenmann invariant to either side of the equation.  What if we consider links instead of knots? In trying to generalize to this case, we will be lead to restricting to a special set of links called proper links.  What if we do not want to assume the boundary 3-manifold is an integral homology sphere but rather a general orientable 3-manifold? In this case, we are lead to a generalization of the Arf invariant known as the Brown invariant and we will need to consider pin structures (to warm up to this, we will first recast some of the earlier results in terms of spin structures).  

In section \ref{sec:lemmas}, we review the algebra underlying the Arf invariant and discuss its applications to invariants of knots in integral homology spheres and characteristic surfaces in 4-manifolds.  We state the version of Rochlin's result that we will later use, and prove a lemma that shows that two Arf invariants are equal.  We also review the definition of the Rochlin invariant of an integral homology sphere.  

In section \ref{sec:main}, we prove our main result and then, by considering several special cases, we recover several results from the literature.  In section \ref{sec:top}, we discuss how to extend these results to the topological category.  In section \ref{sec:links}, we dwell a bit on the algebra underlying the Arf invariant in order to justify the assumption that we only consider proper links (this condition will arise again when discussing spin and pin structures).  After this, we mention how to extend our results to the case of proper links.  In section \ref{sec:brown}, we introduce the Brown invariant of a link in an integral homology sphere, closely following the construction of this invariant in $S^3$ by Kirby and Melvin in \cite{kirby_melvin}, and then we prove a relative version of the nonorientable analogue of Rochlin's theorem of Guillou and Manin \cite{guillou_manin}.  In section \ref{sec:spin}, we recast some of our previous results in the language of spin structures and give alternative proofs of some of these results using this language as a warm-up to section \ref{sec:pin}.  In section \ref{sec:pin}, we recast our results concerning the Brown invariant in terms of work of Kirby and Taylor \cite{kirby1990pin}, in particular showing that the invariant introduced in section \ref{sec:brown} is in fact a special case of an invariant that appears in \cite{kirby1990pin}.  Then, using the language of pin structures, we are able to formulate and prove a generalization of all of the previous results where the 3-manifold boundary is no longer assumed to be an integral homology sphere.  

We summarize our basic proof strategy as ``cap off the 4-manifold and surface, use a result for closed manifolds, compute using additivity, and use a result relating a 4-dimensional invariant to a 3-dimensional invariant".  This strategy is used again and again in the proofs of Theorems \ref{main_thm}, \ref{main_thm_top}, the version of Theorem \ref{main_thm_top} for links in section \ref{sec:links}, Theorem \ref{brown_thm}, the alternative proof of Theorem \ref{main_thm} in section \ref{sec:spin}, the alternative proof of Theorem \ref{brown_thm} in section \ref{sec:pin}, and Theorem \ref{thm:pin}.   For each different proof, we need an appropriate result for closed manifolds (see Theorems \ref{matsumoto}, \ref{thm:closed_top}, \ref{gm}, \ref{kirby_rochlin}, and  \ref{final_closed}) as well as the appropriate result relating 4-dimensional and 3-dimensional invariants (see Lemmas \ref{3d-4d}, \ref{proper_relative_lemma}, \ref{brown_relative_lemma}, \ref{rob:arf}, \ref{lemma:final3d4d}, and \ref{lemma:pinnified}).  

As the paper progresses, we require less and less of our 4-manifolds $X$ and characteristic surfaces $F$, and this is reflected in the fact that our invariants become increasingly complicated.

\section{The Arf invariant}\label{sec:lemmas}

In this section, we briefly review the definitions of the Arf invariant of a quadratic form, the Arf invariant of a knot inside of a homology sphere, and the Arf invariant of a (not necessarily closed) characteristic surface in a 4-manifold with trivial integral first homology.  We then derive a relationship between these invariants that will be used to prove our result.  We also recall the definition of the Rochlin invariant of a homology sphere.  Manifolds will be connected unless stated otherwise.  

Let $V$ be a finite-dimensional $\mathbb{Z}/2\mathbb{Z}$-vector space.  A function $q : V \to \mathbb{Z}/2\mathbb{Z}$ is called a quadratic form if the function $I : V \otimes V \to \mathbb{Z}/2\mathbb{Z}$ defined by
$$
I(x,y) = q(x+y) + q(x) + q(y)
$$
is a bilinear form.  If this form is nondegenerate, then we say that $q$ is nondegenerate.  If $q$ is nondegenerate, then there exists a basis $a_1,b_1,...,a_n,b_n$ for $V$ so that $I(a_i,a_j) = I(b_i,b_j) = 0$ and $I(a_i,b_j) = \delta_{ij}$ for all $i,j$.  The Arf invariant of $q$, denoted $\Arf(q)$ is defined as 
$$
\Arf(q) = \sum_{i=1}^n q(a_i)q(b_i)
$$
and does not depend on the choice of such a basis.  Two such vector spaces $V,W$ with respective nondegenerate quadratic forms $q_V$ and $q_W$ are equivalent if there exists an isomorphism of vector spaces $\phi : V \to W$ such that $q_W \circ \phi = q_V$.  Arf proved that $(V,q_V)$ and $(W,q_W)$ are equivalent if and only if $\dim(V) = \dim(W)$ and $\Arf(q_V) = \Arf(q_W)$ \cite{arf1941}.  The Arf invariant is additive in the sense that if we have such a pair $(V,q)$ where $q$ is nondegenerate, and $V$ splits as an orthogonal direct sum 
$$
V = V_1 \oplus V_2
$$
then 
$$
\Arf(V,q) = \Arf(V_1, q|_{V_1}) + \Arf(V_2, q|_{V_2})
$$ 
(see for example, \cite{saveliev_book}).   

For any knot $K$ in an integer homology sphere $\Sigma$, there is a $\mathbb{Z}/2\mathbb{Z}$ valued cobordism invariant called the Arf invariant of $K$, denoted $\Arf(K)$.  Let $S$ be a Seifert surface for $K$.  Then we have a quadratic form $q_S : H_1(F; \mathbb{Z}/2\mathbb{Z}) \to \mathbb{Z}/2\mathbb{Z}$ defined by taking an element $x$ in $H_1(F;\mathbb{Z}/2\mathbb{Z})$, representing it by an embedded curve $C \subset F$, and then taking $q_F(x) = \lk(C, C^+) \pmod 2$ where $C^+$ is the result of pushing $C$ off of $F$ using chosen nonvanishing section of the normal bundle.  The resulting bilinear form 
$$
I_S : H_1(S;\mathbb{Z}/2\mathbb{Z}) \otimes H_1(S;\mathbb{Z}/2\mathbb{Z}) \to \mathbb{Z}/2\mathbb{Z}
$$
defined by 
$$
I_S(x,y) = q_S(x+y) + q_S(x) + q_S(y)
$$
is the intersection form on $S$.  The Arf invariant of $K$ is then defined as $\Arf(K) = \Arf(q_S)$.  The fact that the result does not depend on the choice of the Seifert surface $S$ follows from the fact that any two Seifert surfaces for a given knot in a homology sphere differ by a sequence of isotopies, additions of tubes, and deletions of tubes.  None of these operations change the Arf invariant of the surface.  The Arf invariant in $S^3$ was originally defined in a different way by Robertello \cite{robertello} where he also shows that the definition is equivalent to the one given here (see the discussion below Theorem \ref{main_thm}).  In \cite{robertello}, only knots in $S^3$ are considered; however, the relevant results also apply in integer homology spheres (see \cite{saveliev_book}).   

We now review the relevant 4-dimensional background.  Let $X^4$ be a smooth compact connected 4-manifold and let $F$ be an connected surface properly embedded in $X$ with connected boundary.  Then $F$ is said to be \emph{characteristic} if $[F] \cdot G = G \cdot G \pmod 2$ for all $G \in H_2(X; \mathbb{Z}/2\mathbb{Z})$ using the pairing $H_2(X, \partial X; \mathbb{Z}/2\mathbb{Z}) \times H_2(X; \mathbb{Z}/2\mathbb{Z}) \to \mathbb{Z}/2\mathbb{Z}$ coming from Lefschetz duality together with the usual pairing on homology, or equivalently if $[F, \partial F]$ is dual to $w_2(X)$ (i.e.,  $[F, \partial F]$ and $w_2(X)$ are identified when we first consider $[F, \partial F] \in H_2(X, \partial X; \mathbb{Z}/2\mathbb{Z})$ and then apply the Lefschetz duality isomorphism $H_2(X, \partial X; \mathbb{Z}/2\mathbb{Z}) \cong H^2(X; \mathbb{Z}/2\mathbb{Z})$).  

We also assume for now that $H_1(M; \mathbb{Z}) = 0$ and that $F$ is orientable -- the case where $F$ is not orientable will be treated in section \ref{sec:brown}. We now define a quadratic form
$$
q_F : H_1(F; \mathbb{Z}/2\mathbb{Z}) \to \mathbb{Z}/2\mathbb{Z}
$$
whose corresponding bilinear form 
$$
I_F(x,y) = q_F(x+y) - q_F(x) + q_F(y)
$$
is the intersection form on $H_1(F; \mathbb{Z}/2\mathbb{Z})$.  Since we will take the Arf invariant of $q_F$, we need to restrict to $F$ having a connected boundary, otherwise the intersection form is degenerate.  We revisit this point in section \ref{sec:links}.  

Given $x \in H_1(F; \mathbb{Z}/2\mathbb{Z})$, first pick an immersed curve $C \subset F$ representing $x$.  Since $H_1(X; \mathbb{Z}) = 0$, we can also pick an orientable surface $D$ bounding $C$ inside of $X$ that is transverse to $F$.  Note that the normal bundle of $D$ in $X$ is orientable and therefore trivial (since $D$ is homotopy equivalent to a wedge of circles) and, in fact, the trivialization that is induced on $C = \partial D$ by restricting a trivialization of the  normal bundle of $D$ to its boundary is unique.  This is because any two trivializations of the normal bundle of $\partial D$ differ by a map $\phi: \partial D \to \SO(2)$ and if these maps both come from trivializations of the normal bundle of $D$ then $\phi$ extends to a map $\Phi : D \to \SO(2)$.  But $\pi_1(\SO(2)) = H_1(\SO(2); \mathbb{Z})$ and therefore these two trivializations are the same.  The normal bundle of $C$ in $F$ then determines a 1-dimensional subbundle of this trivialized 2-dimensional bundle over $C$ and we let $\mathcal{O}(D)$ be the mod 2 number of twists in this 1-dimensional subbundle as we go around $C$.  Now we define 
$$
q(x) =  D \cdot F + \mathcal{O}(D) + d(C) \pmod 2
$$
where $d(C)$ is the number of double points of $C$ and $D \cdot F$ is the number of intersection points between $D$ and $F$.    To see the importance of $F$ being characteristic, imagine changing the above choice of $D$ by tubing into some closed surface.  

In \cite{matsumoto}, it is shown that $q_F$ is well defined, and further that if $[F] = [G]$ in $H_2(X, \partial X; \mathbb{Z})$ and $\partial F$ and $\partial G$ are concordant, then $\Arf(q_F) = \Arf(q_G)$. We denote this also by $\Arf(F) = \Arf(q_F)$.  In particular, when $F$ is closed, then for any other orientable surface $G$ homologous to $F$, $\Arf(q_F) = \Arf(q_G)$ and we denote the result simply by $\Arf([F])$ where $[F]$ is the corresponding homology class in $H_2(X; \mathbb{Z})$.  

Recall that for any symmetric unimodular bilinear form on a finitely generated free abelian group, there are characteristic elements and the square of any characteristic element is congruent to the signature of the form modulo 8 (see, for example, \cite{kirbybook}).  In particular, this applies to the intersection form on any closed orientable 4-manifold or any compact orientable 4-manifold with integral homology sphere boundary.  Rochlin proved the following result which we use in deriving our relative result \cite{rokhlin1972proof} (see also \cite{matsumoto}):

\begin{theorem}(Rochlin) \label{matsumoto}
	Let $X^4$ be a closed oriented smooth 4-manifold with $H_1(X; \mathbb{Z}) = 0$, and let $\xi$ be a characteristic homology class in $H_2(X; \mathbb{Z})$.  Then
	$$
	\Arf(\xi) = \frac{\sigma(X) - \xi \cdot \xi}{8} \pmod 2	
	$$
\end{theorem}

The condition that $H_1(X;\mathbb{Z}) = 0$ in Theorem \ref{matsumoto} can be defined away as follows.  We start with the setup of the closed case as in Theorem \ref{matsumoto}, but we do not insist that $H_1(X; \mathbb{Z}) = 0$.  Represent $\xi$ by an embedded orientable surface $F$.  Let $L$ be a 1-dimensional link embedded in $X - F$ so that the components of $L$ generate $H_1(X;\mathbb{Z})$.  By doing surgery on $L$, we obtain a new 4-manifold $X'$ and we can consider $F$ as being in $X'$ as well.  In this case, Theorem \ref{matsumoto} applies and says that 
	\begin{equation} \label{eqn}
	\Arf([F]) = \frac{\sigma(X') - [F]^2}{8} \pmod 2	
	\end{equation}
	Note that $\sigma(X') = \sigma(X)$ since $X$ and $X'$ are cobordant and $\xi \cdot \xi = [F]^2$ since they have isomorphic normal bundles.  We define $\Arf(\xi) := \Arf([F])$ for some choice of $F$ and $L$ where $\Arf([F])$ is considered as being in $X'$.  Note that the choice of $L$ does not matter as the left hand side of equation \ref{eqn} does not depend on $L$.  Also the choice of $F$ representing $\xi$ does not matter as given a different choice, say $F'$, we necessarily have $[F']^2 = [F]^2$ and we can choose $L$ to be disjoint from $F$ and $F'$.  Then using $F$ and $L$ we obtain 
	$$
	\Arf(\xi) = \frac{\sigma(X) - [F]^2}{8} \pmod 2	
	$$
and using $F'$ and $L$, we obtain 
	$$
	\Arf(\xi) = \frac{\sigma(X) - [F']^2}{8} \pmod 2	
	$$
	Therefore, $\Arf(\xi)$ is well defined.  In the case of Theorem \ref{main_thm} where $X$ has boundary, we can similarly define $\Arf(F)$ when $H_1(X;\mathbb{Z}) \neq 0$ by choosing an $L$ disjoint from $F$  and proceeding as above.  With these definitions in place, Theorem \ref{matsumoto} holds without the hypothesis that $H_1(X;\mathbb{Z}) = 0$ and it is in this form that we use the previous result.

Note that by the above mentioned results, if we are given a knot in a homology sphere and an orientable surface $F$ whose boundary is $K$ contained in $\Sigma \times I$, we have two $\mathbb{Z}/2\mathbb{Z}$-valued concordance invariants associated to $K$, $\Arf(K)$ and $\Arf(F)$.  Note that for any other such surface $F'$, we have $\Arf(F) = \Arf(F')$.  These invariants coincide:

\begin{lemma} \label{3d-4d}
Let $\Sigma^3$ be a integer homology sphere and let $F^2$ be a connected orientable properly embedded surface in $\Sigma \times I$ with connected boundary.  Then $\Arf( \partial F) = \Arf(F)$.
\end{lemma}

\begin{proof}
	Assume that we start with a Seifert surface $S$ for $\partial F$ in $\Sigma = \Sigma \times \{0\}$ and let $C$ be an embedded curve in $F$ that is nontrivial in $H_1(S; \mathbb{Z}/2\mathbb{Z})$.  Let $S'$ be the result of pushing $S$ into $\Sigma \times I$ so that it is properly embedded.  Then $\Arf(F) = \Arf(S')$ by our previous discussion.  Thus, it suffices to check that the 3-dimensional Arf invariant $\Arf(\partial F)$ is equal to the 4-dimensional $\Arf(S')$.  We do this by showing that the two quadratic forms $q_S$ and $q_{S'}$ agree.  We denote the push-in of the curve $C$ by $C'$.  Suppose that the particular push-in $S'$ has been chosen so that $C'$, together with a small annular neighborhood, is embedded in a particular time slice of $\Sigma \times \{t_0\}$ and that this is the time slice with the maximum time coordinate that intersects $S'$ (note that all push-ins of $S$ are isoptic relative to their boundaries and so we can pick whichever push-in is convenient). 

	Let $(C')^+$ be the knot in $\Sigma \times \{t_0\}$ obtained by pushing $C'$ off of the annular neighborhood in $\Sigma \times \{t_0\}$ and let $D$ be a Seifert surface for $(C')^+$ in $\Sigma \times \{t_0\}$.  Then, when computing the 3-dimensional quadratic form, we have $q_S(C) = \lk(C, C^+) = |C \cap D|$ (where here, we have projected $D$ from $\Sigma \times \{t_0\}$ to $\Sigma \times \{0\}$ and $C^+$ is $(C')^+$ projected from $\Sigma \times \{t_0\}$ to $\Sigma \times \{0\}$).  In computing the 4-dimensional quadratic form $q_{S'}(C')$, we use the surface $D$ together with the annulus between $C'$ and $(C')^+$, which we will simply refer to as $D$.  Since $C$ is embedded, we have $d(C) = 0$.  To see that $\mathcal{O}(D) = 0$, note that the framing of the normal bundle of $D$ is given by one vector field pointed entirely in the time direction, and one vector field entirely normal to $D$ in the time slice $\Sigma \times \{t_0\}$, which, along the annulus, agrees with a normal framing of $C$ within the annulus. Finally, after perturbing $S'$ slightly so that the only $C$ appears in $\Sigma \times \{t_0\}$ by pushing the rest of the annulus slightly up in the time direction, we see that $D \cdot S'$ is the number of intersections between $D$ and $C'$, which is just $\lk(C, C^+)$.  Therefore, $q_S(C) = q_{S'}(C')$, and since we can choose a symplectic basis for $H_1(S; \mathbb{Z}/2\mathbb{Z})$ to consist of such embedded curves, it follows that $\Arf(\partial F) = \Arf(F)$.  
\end{proof}

Finally, we mention the Rochlin invariant of a 3-dimensional integer homology sphere $\Sigma$, which is defined as
$$
\mu(\Sigma) = \frac{\sigma(W)}{8} \pmod 2 
$$
where $W$ is a smooth even compact 4-manifold with $H_1(W; \mathbb{Z}) = 0 $ that bounds $\Sigma$ (see, for example, \cite{saveliev_book}).  Note that we need to orient $W$ to define $\sigma(W)$; however, changing the orientation does not alter the result modulo 2.  For the moment, choose an orientation on $W$ and hence $\Sigma$. If $W'$ where another such oriented smooth even 4-manifold with $H_1(W';\mathbb{Z}) = 0$ with boundary $\overline{\Sigma}$, then we can form the closed oriented even smooth 4-manifold $W \cup_\Sigma W'$, and by applying Theorem \ref{matsumoto} with $\xi = 0$, we obtain
$$
0 = \frac{\sigma(W \cup_\Sigma W')}{8} = \frac{\sigma(W)}{8} + \frac{\sigma(W')}{8} \pmod 2
$$
by Novikov additivity \cite{atiyah1968index}, and therefore, $\mu(\Sigma)$ is well defined.  Note that $\mu$ is invariant under integral homology concordances of homology spheres.

\section{A relative version of Rochlin's theorem} \label{sec:main}

In this section, we put together the pieces from the last section, prove a relative version of Rochlin's theorem, and then discuss the relationship of this theorem to some of the known results in the literature.  The argument used to prove the following theorem will be repeated and referred to several times in the sections that follow, in which we will make minor adjustments to the proof in order to prove various generalizations.  

First we make some homological observations.  Let $X^4$ be a compact connected oriented 4-manifold whose boundary $\partial X$ is an integral homology sphere.  The relative long exact sequence for the boundary then gives the isomorphism $H_2(X; \mathbb{Z}) \cong H_2(X, \partial X; \mathbb{Z})$, and similarly in cohomology we have $H^2(X; \mathbb{Z}) \cong H^2(X, \partial X; \mathbb{Z})$.  We also have the isomorphism $H_2(X; \mathbb{Z}) \cong H^2(X, \partial X; \mathbb{Z})$.  By composing the isomorphisms $H^2(X, \partial; \mathbb{Z}) \cong H_2(X; \mathbb{Z}) \cong H^2(X, \partial X; \mathbb{Z})$ the pairing $H^2(X, \partial X; \mathbb{Z}) \times H_2(X, \partial X; \mathbb{Z}) \to \mathbb{Z}$ yields a nondegenerate symmetric pairing $H_2(X, \partial X; \mathbb{Z}) \times H_2(X, \partial X; \mathbb{Z}) \to \mathbb{Z}$ which is geometrically given by counting intersection points with sign.  For orientable $F$, the condition that $F$ is characteristic is equivalent to $[F]$ being characteristic with respect to this pairing.  When we write $[F]^2$ below, we are using this pairing.  Geometrically, this is obtained by pushing $F$ off of itself using the 0-framing on all of the components of $\partial F$ in $\partial X$, and extending this to a push-off of $F$ in $X$, and counting the signed number of intersections between $F$ and this push-off.   

\begin{theorem} \label{main_thm}
	Let $X^4$ be a smooth compact connected oriented 4-manifold with $\partial X$ an integer homology sphere.  Let $F^2$ be an orientable characteristic surface with connected boundary that is properly embedded in $X$.  Then
	$$
	\Arf(F) + \Arf(\partial F) = \frac{\sigma(X) - [F]^2}{8} + \mu(\partial X)  \pmod 2
	$$
\end{theorem}

\begin{proof}
	Assume for now that $H_1(X;\mathbb{Z}) = 0$.  Let $W^4$ be a smooth even compact oriented 4-manifold with $H_1(W; \mathbb{Z}) = 0$ and $\partial W = \overline{\partial X}$.  Let $F_W$ be an orientable surface properly embedded $W$ inside a collar neighborhood $\partial X \times I$ with $\partial F_W = \partial F$.  Since $W$ is even, the closed $F \cup F_W$ is characteristic in the closed 4-manifold $X \cup W$.  See Figure \ref{fig:slice_schematic} for a schematic.  Note that by Novikov additivity \cite{atiyah1968index}, $\sigma(X \cup W) = \sigma(X) + \sigma(W)$.  

	Since $\Sigma$ is a homology sphere, $\partial F$ bounds an oriented surface in $\partial X$.  Pick a bicollared neighborhood $\partial X \times [-1,1]$ embedded in $X \cup W$ with $\partial X \times \{0\}$ mapped to $\partial X$ via the identity, $\partial X \times [-1,0]$ contained in $X$, and $\partial X \times [0,1]$ contained in $W$.  Further, assume that we have isotoped $F$ and $F_W$ relative to their boundaries so that inside of  $\partial X \times [-1,0]$ and $\partial X \times [0,1]$, they intersect each slice $\partial X \times \{t\}$ in just a copy of $\partial F$.  Then let $\overline{F}$ be the part of $F$ outside of the collar, together with a copy of a Seifert surface for the $F \cap (\partial X \times \{-1\})$ pushed in to $\partial X \times \{[-1,0)\}$.  Similarly, let $\overline{F_W}$ be the part of $F_W$ outside of the collar, together with a copy of a Seifert surface for the $F \cap (\partial X \times \{1\})$ pushed in to $\partial X \times \{(0,1]\}$. For the homology class of $F \cup F_W$ in $H_2(X \cup W; \mathbb{Z})$ (where now an orientation has been chosen) we have 
	$$
	[F \cup F_W] = [\overline{F}] + [\overline{F_W}] = [\overline{F}]
	$$ 
	and 
	$$
	[\overline{F}] \cdot [\overline{F_W}] =0
	$$  
	Therefore, we have 
	$$
	[F \cup F_W]^2 = [F]^2
	$$  

	Finally, $\Arf(F \cup F_W)$ can also be computed in two parts.  Namely, since both $X$ and $W$ have vanishing first integral homology, the quadratic space $(H_1(F \cup F_W; \mathbb{Z}/2\mathbb{Z}), q_{F \cup F_W})$ splits as an orthogonal direct sum 
	$$
	(H_1(F \cup F_W; \mathbb{Z}/2\mathbb{Z}), q_{F \cup F_W}) = (H_1(F; \mathbb{Z}/2\mathbb{Z}), q_F) \oplus (H_1(F_W; \mathbb{Z}/2\mathbb{Z}), q_{F_W})
	$$
and therefore, by the additivity of the Arf invariant, we have 
	$$
	\Arf(F \cup F_W) = \Arf(F) + \Arf(F_W) = \Arf(F) + \Arf(\partial F)
	$$
where the last equality follows from Lemma \ref{3d-4d} since $F_W$ is contained in a collar neighborhood of the boundary in $W$.  

Putting these additivity results together and applying Theorem \ref{matsumoto}, we obtain
$$
	\Arf(F) + \Arf(\partial F) = \Arf(F) = \frac{\sigma(X \cup W) - [F \cup F_W]^2}{8}=  \frac{\sigma(X) - [F]^2}{8} + \mu(\partial X)  \pmod 2
$$
as desired. 

	The case where $H_1(X;\mathbb{Z}) \neq 0$ is obtained by doing surgery on a link generating $H_1(X;\mathbb{Z})$ in the complement of $F$ and applying the case where $H_1(X;\mathbb{Z}) = 0$. 

\end{proof}

\begin{figure}  	
	\centering
	\includegraphics[width=5cm]{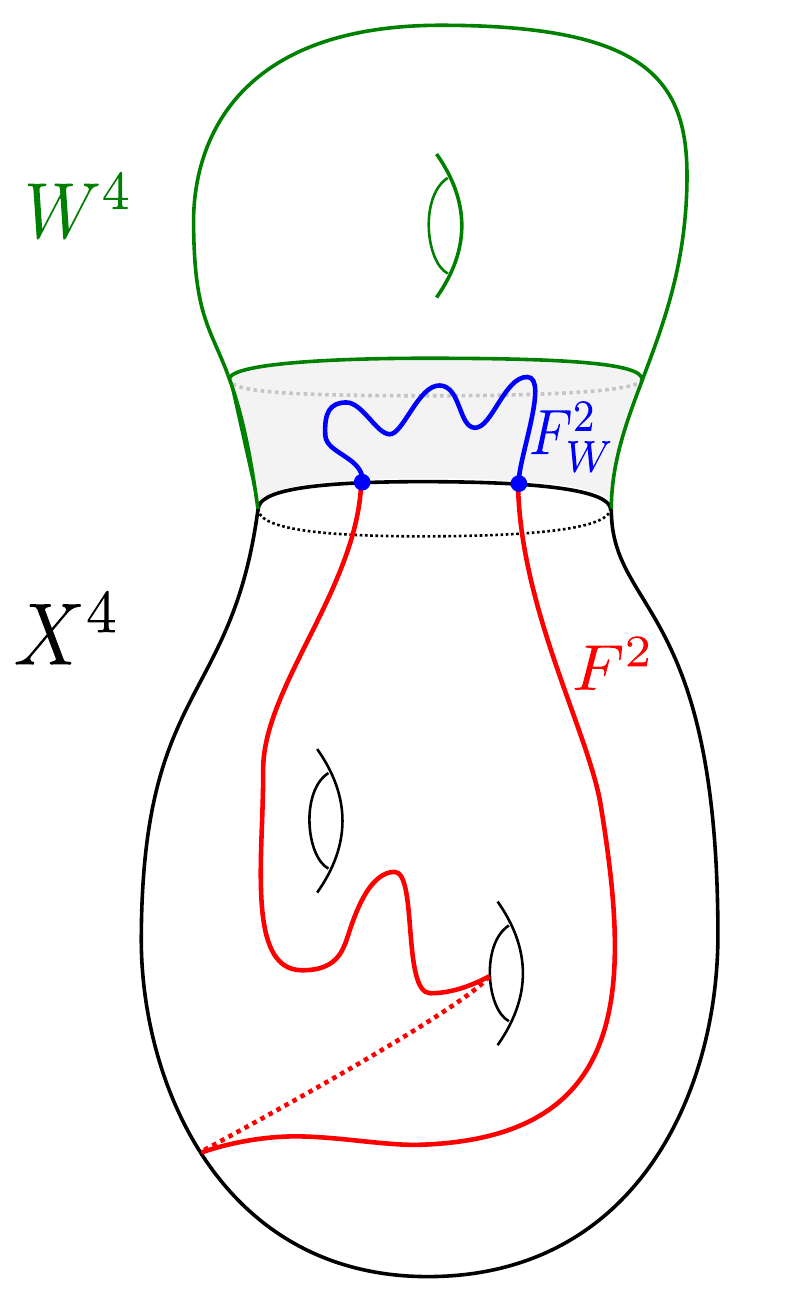}
  	\caption{This figure is a schematic for our situation where $X^4$ is the 4-manifold containing the characteristic surface $F^2$, $W^4$ is the manifold that also bounds $\partial X$, which we have glued on to obtain a closed 4-manfiold, and $F^2_W$ is a surface in a collar neighborhood of $\partial W$ with $\partial F_W = \partial F$.} \label{fig:slice_schematic}
\end{figure}

There are other results relating the Rochlin invariant and the Arf invariant.  Ac\~una-Gonz\'alez proved that, given a knot $K \subset S^3$ and $\alpha \in \mathbb{Z}$ with $S^3(K; \alpha)$ equal to the integral homology sphere that results from $(1/\alpha)$-surgery on $K$, we have $\mu(S^3(K; \alpha)) = \alpha \cdot \Arf(K)$ \cite{gonzalez1970dehn}.  We recover this result when $\alpha = \pm 1$ from Theorem \ref{main_thm} as follows.  Consider a copy of the mirror image of $K$ contained in a 3-ball in $S^3$ that is disjoint from $K$. We denote the knot by $K^m$.  Then $\Arf(K^m) = \Arf(K)$ and there exists an annulus properly embedded in $B^4$ that cobounds $K$ and $K^m$.  Now attach a 2-handle to $B^4$ with framing $\alpha$, denote the resulting 4-manifold be $X$, and let $F$ denote the aforementioned annulus together with the core of the 2-handle.  Note that $F$ is characteristic in $X$ and that the Arf invariant of $K^m$ considered in $\partial X$ is the same as the Arf invariant of $K^m$ considered in $S^3$, since this can be computed from a Seifert surface and the whole computation can be assumed to happen within the 3-ball containing $K^m$.  Then, applying Theorem \ref{main_thm}, we see that $[F]^2 = \sigma(X)$ and $\Arf(F) = 0$, thus rederiving Ac\~una-Gonz\'alez's result in this case.  

Gordon has generalized this result of Ac\~una-Gonz\'alez to the case where two solid torus exteriors in integer homology spheres are glued together to obtain a $\mathbb{Z}/2\mathbb{Z}$-homology sphere \cite{gordon1975knots}.  In particular, Gordon derived a formula for the Rochlin invariant of such a homology sphere in terms of the gluing map, the Arf invariants of the cores of the removed solid tori, and the Rochlin invariants of the two homology spheres \cite{gordon1975knots}.

\section{A relative version of Rochlin's theorem for topological manifolds}\label{sec:top}

In this section, we discuss how to extend the results in the previous sections to the topological category.  Let $\KS(X)$ be the Kirby-Siebenmann invariant of a connected topological 4-manifold \cite{kirby1977foundational}.  The starting point for this discussion is a theorem of Freedman and Kirby (as stated in \cite{universal_form}, although we drop the hypothesis that the manifold be simply-connected as addressed below the theorem statement). This result generalizes Theorem \ref{matsumoto} to the topological category:

\begin{theorem}(Freedman and Kirby)\label{thm:closed_top}
	Let $X^4$ be a closed oriented topological 4-manifold, and let $\xi$ be a characteristic homology class in $H_2(X; \mathbb{Z})$.  Then
	$$
	\Arf(\xi) = \KS(X) + \frac{\sigma(X) - \xi \cdot \xi}{8} \pmod 2	
	$$
\end{theorem}

Note that in \cite{universal_form}, the invariant $\tau = \tau_1$ is used in place of $\Arf$, however, these two invariants are equal (see Lemma 6 in \cite{universal_form}).  This theorem follows from Theorem \ref{matsumoto} together with the topological result of Freedman given by the equivalent criteria (iv) and (v) in Theorem 3 of \cite{universal_form}.  

By using the definition of $\Arf(\xi)$ for non-simply-connected manifolds, as mentioned in Section \ref{sec:lemmas}, together with the fact that $KS(M) = KS(N)$ if $M$ and $N$ are concordant (see, for example, Theorem 8.2 of \cite{2020survey}), Theorem \ref{thm:closed_top} with no hypothesis on $\pi_1(X)$ follows from the version of Theorem \ref{thm:closed_top} stated in \cite{universal_form} (which assumes $X$ is simply-connected) in the exact same way that we generalized Theorem \ref{matsumoto} with no condition on $H_1(X;\mathbb{Z})$ in Section \ref{sec:lemmas}.

\begin{remark}
The invariants $\tau$ and $\Arf$ are both obstructions to representing a homology class by an embedded sphere.  From the perspective of $\tau$, we represent a homology class by an immersed sphere and then try and eliminate the double points via Whitney moves and $\tau$ is an obstruction to being able to do this.  From the perspective of $\Arf$, we start with a higher genus surface that represents a given homology class and we try and find disks to do compressions that lower the genus of the surface and $\Arf$ is an obstruction to doing this. 
\end{remark}

We will make use of two properties of $\KS$ (see, for example, Theorem 8.2 of \cite{2020survey}).  First, if $X^4$ and $W^4$ are two compact 4-manifolds with $\partial X$ homeomorphic to $\partial W$, then $\KS(X \cup_\phi W) = \KS(X) + \KS(W)$ where $\phi : \partial X \to \partial W$ is any homeomorphism.  Second, if $W$ is a smoothable compact 4-manifold, then $\KS(W) = 0$.  Using Theorem \ref{thm:closed_top} together with the aforementioned facts about $\KS$ applied to $W$ as in the proof of Theorem \ref{main_thm}, the proof of Theorem \ref{main_thm} extends to the topological category and yields:

\begin{theorem} \label{main_thm_top}
Let $X^4$ be a compact oriented topological 4-manifold with $\partial X$ an integer homology sphere.  Let $F^2$ be an orientable characteristic surface with connected boundary that is properly embedded in $X$.  Then
	$$
	\Arf(F) + \Arf(\partial F) =  \frac{\sigma(X) - [F]^2}{8} + \mu(\partial X) + \KS(X) \pmod 2
	$$
\end{theorem}

\section{A relative version of Rochlin's theorem for proper links} \label{sec:links}

In this section, we discuss the Arf invariant of a general link in a homology sphere and highlight the need to restrict to \emph{proper} links (see below).  We then indicate how the results of the previous section extend to this setting.

The Arf invariant for an oriented link $L$ in a homology sphere $\Sigma$ is generally only defined when the link is so called \emph{proper}, which means that for any component $K$ of the link $L$, the linking number $\lk(K, L-K)$ is even.  Robertello showed that if $K_1$ and $K_2$ are two knots in a homology sphere such that there are planar surfaces $P_1$ and $P_2$ properly embedded in $\Sigma \times I$ with 
$$
\partial P_1 = (L \times \{0\}) \cup (K \times \{1\})$$ 
and 
$$\partial P_2 = (L \times \{0\}) \cup (K_2 \times \{1\})$$
then $\Arf(K_1) = \Arf(K_2)$.  This gives a definition of the Arf invariant of a proper link $L$ as in, for example, \cite{hoste1984arf}.  By considering, say, the Hopf link, the necessity of the properness condition is seen.  There is a different definition similar to the 3-dimensional definition that we gave earlier for the Arf invariant of a knot (see, for example, \cite{lickorish1997introduction}).  

Here, we review this definition of the Arf invariant for links, following the presentation of the closely related Brown invariant of a link in \cite{kirby_melvin}, which, we believe, illuminates why the Arf invariant is only defined for proper links.  Lickorish offers a different explanation for this in terms of the Jones polynomial for links evaluated at $t = i$ \cite{lickorish1997introduction}.  We discuss the relevant 4-dimensional concepts, relate them to the 3-dimensional ideas, and prove a relative version of Rochlin's theorem for proper links.  This is a direct and straightforward generalization of what we did in the previous sections.  

We now review the algebra to illuminate the role of properness.  Let $V$ be a finite-dimensional vector space over $\mathbb{Z}/2\mathbb{Z}$ and let $\cdot : V \otimes V \to \mathbb{Z}/2\mathbb{Z}$ be a symmetric bilinear, but not necessarily nondegenerate form on $V$.  Let $R = \{ r \in V : r \cdot v = 0 \text{ for all } v \in V \}$ be the radical of $(V, \cdot)$.  Suppose further that there is a map $q : V \to \mathbb{Z}/\mathbb{Z}$ such that 
$$
q(x + y) = q(x) + q(y) + x \cdot y
$$
for all $x,y \in V$.  We will call $(V, \cdot, q)$ a \emph{quadratic space}.  If $q|_R = 0$, then we call $(V, \cdot, q)$ \emph{proper}.  In this case, we have a quadratic space $(V/R, \cdot, q)$ with a nondegenerate inner product $\cdot$, and therefore, we can use the previous definition and define the Arf invariant of $(V,\cdot,q)$ to be $\Arf(V/R, \cdot, q)$.  If $(V, \cdot, q)$ is not proper, then we make the convention that $\Arf(V, \cdot, q) = \infty$.  Together, the three quantities $\dim(V), \dim(R)$, and $\Arf(V,\cdot,q)$ completely classify $(V,\cdot,q)$ up to isomorphism.  Furthermore, the Arf invariant is additive with respect to sums of quadratic spaces, and since the sum of a nonproper quadratic space with any other quadratic space is again nonproper, this justifies the $\infty$ notation.  

The ``democratic" method of computing Arf also extends to this setting:

\begin{proposition}
	Let $(V, \cdot, q)$ be a quadratic space and let $q_0$ and $q_1$ be the cardinalities of the preimage of $0$ and $1$ under $q$, respectively.  Then
	\[
  \Arf(V, \cdot, q) =
  \begin{cases}
	  0 & \text{ if $q_0 > q_1$}, \\
	  1 & \text{ if $q_0 < q_1$},\\
	  \infty & \text{ if $q_0 = q_1$} 
  \end{cases}
	\]
\end{proposition}

\begin{proof}
	If $(V, \cdot, q)$ is not proper, then there exists an element $r$ in the radical with $q(r) = 1$.  Writing $V = (\mathbb{Z}/2\mathbb{Z})r \oplus W$, we find that given any $w \in W$, $q(w) \neq q(w + r)$ and therefore $q_0 = q_1$.  Now we assume that $(V, \cdot, q)$ is proper.  If the radical $R$ is trivial, then this result is shown in, for example, \cite{saveliev_book}.  If $R \neq 0$, then we will consider the restriction of $q$ to the quotient $V/R$, which we denote by $q_{V/R}$.  Then, letting $(q_{V/R})_0$ and $(q_{V/R})_1$ denote the cardinalities of the preimages of $0$ and $1$, respectively, we have $(q_{V/R})_0 = |R| q_0$ and $(q_{V/R})_1 = |R| q_1$, from which the result follows.  
\end{proof}

We now return to topology, in particular 3-dimensional topology.  Let $L$ be an oriented link in an integer homology sphere $\Sigma$. Let $F$ be a Seifert surface for $L$ and define as before
\begin{align*}
	q_S : H_1(S; \mathbb{Z}/2\mathbb{Z}) &\to \mathbb{Z}/2\mathbb{Z} \\
	x &\mapsto \lk(C, C^+)
\end{align*}
where $C$ is a curve embedded in $S$ representing $x$ and $C^+$ is a push-off of $a$ from $F$ in the normal direction.  As before, we have 
$$
q_S(x + y) = q_S(x) + q_S(y) + x \cdot y
$$
for all $x,y \in H_1(S; \mathbb{Z}/2\mathbb{Z})$, where here $\cdot$ is the intersection form on $H_1(S;\mathbb{Z}/2\mathbb{Z})$.  Note that $\im(H_1(\partial S;\mathbb{Z}/2\mathbb{Z}) \to H_1(S;\mathbb{Z}/2\mathbb{Z}))$ is the radical of this quadratic space and further note that 
\begin{align*}
	q_S([K]) &= \lk(K, K^+) \\ 
	         &= \lk(K, L-K) 
\end{align*}
where $K$ is a component of $L$ and $K^+$ is a push-off of $K$ using the surface framing.  Therefore, the quadratic space $(H_1(S;\mathbb{Z}/2\mathbb{Z}), \cdot, q_S)$ is proper if and only if $L$ is proper.  We define $\Arf(L) = \Arf(H_1(S; \mathbb{Z}/2\mathbb{Z}, \cdot, q_S)$ where, as in the case of knots, this is independent of the choice of Seifert surface (see \cite{lickorish1997introduction}).  

We now move on to 4-dimensions.  The situation is completely analogous to what we did for knots in the last section.  The facts that we need are not exactly found in \cite{matsumoto}, so we state the results here.  However, they follow exactly the corresponding proofs of the statements in \cite{matsumoto} and thus, we omit the proofs.  

\begin{lemma} \label{proper_relative_lemma}
	Let $X^4$ be a compact smooth 4-manifold with $H_1(X; \mathbb{Z}) = 0$ and let $F^2$ be an orientable characteristic surface that is properly embedded in $X$.  Let $x \in H_1(F; \mathbb{Z}/2\mathbb{Z})$ and suppose that $X$ is represented by a curve $C \in F$, and $D^2$ is an immersed orientable surface in $X$ that is transverse to $F$.  Then define 
	\begin{align*}
		q_F : H_1(F;\mathbb{Z}/2\mathbb{Z}) &\to \mathbb{Z}/2\mathbb{Z} \\
		x &\mapsto  D \cdot F + \mathcal{O}(D) + d(C) 
	\end{align*}
	where $\mathcal{O}(D)$ is the framing of $D$, and $d(C)$ is the number of double points of $C$, and $D \cdot F$ is the number of points of intersection between $D$ and $F$.  Then $q_F$ is well defined and, together with the intersection form, makes $H_1(F; \mathbb{Z}/2\mathbb{Z})$ a quadratic space.  If $F_1$ and $F_2$ are two such surfaces with $\partial F_1$ and $\partial F_2$ contained in the same boundary component of $X$ such that they are concordant in that boundary component, and such that the homology classes $[F_1] = [F_2]$ in $H_2(X, \partial X; \mathbb{Z}/2\mathbb{Z})$, then $\Arf(q_{F_1}) = \Arf(q_{F_2})$.  If $\partial F$ is contained in an integer homology sphere's boundary component of $X$, then $q_F$ is proper if and only if $\partial F$ is a proper link.  
\end{lemma}

With this setup, results analogous to Lemma \ref{3d-4d} and Theorem \ref{main_thm}, where now $F$ is allowed to have boundary a proper link, follow exactly as in the case of knots. The topological generalizations of these results also follow in the same manner and the hypothesis that $H_1(X;\mathbb{Z}) =0$ can be defined away just as in Section \ref{sec:lemmas}.

\section{The Brown invariant} \label{sec:brown}

In this section, we discuss the Brown invariant of a $\mathbb{Z}/4\mathbb{Z}$-enhanced inner product space over $\mathbb{Z}/2\mathbb{Z}$ following the general path we have now gone along a few times: first, we review the relevant bit of algebra; second, we discuss the use of this algebra to give an invariants of 3-dimensional knots; third, we discuss 4-dimensional results relating the algebra to characteristic surfaces in 4-manifolds; and finally, we prove a result that brings the 3- and 4-dimensional perspectives together.  

We start with the algebra.  Let $V$ be a finite-dimensional vector space over $\mathbb{Z}/2\mathbb{Z}$ and let $\cdot : V \otimes V \to \mathbb{Z}/2\mathbb{Z}$ be a symmetric bilinear, but not necessarily nondegenerate, form on $V$.    Suppose further that there is a map $e : V \to \mathbb{Z}/4\mathbb{Z}$ such that 
$$
e(x + y) = e(x) + e(y) + 2(x \cdot y)
$$
for all $x,y \in V$ where here $2 : \mathbb{Z}/2\mathbb{Z} \to \mathbb{Z}/4\mathbb{Z}$ is the inclusion map (we use such inclusions going forward without note).  We call such an $e$ a \emph{$\mathbb{Z}/4\mathbb{Z}$-valued quadratic enhancement of $\cdot$}.   We call $(V, \cdot, e)$ an \emph{enhanced space}.  One source of examples of enhanced spaces comes from taking a quadratic space $(V,\cdot,2q)$ where $(V,\cdot, q)$ is a quadratic space. If $e|_R = 0$, then we call $(V, \cdot, e)$ \emph{proper}.  

We now define the Brown invariant of an enhanced space $(V, \cdot, e)$ following \cite{kirby_melvin}.  Let $e_i$ be the size of the preimage of $i \in \mathbb{Z}/4\mathbb{Z}$ with respect to the map $e : V \to \mathbb{Z}/4\mathbb{Z}$.  Then the \emph{Brown invariant} of $(V, \cdot, e)$, denoted $\beta(V, \cdot, e) \in \mathbb{Z}/8\mathbb{Z}$, is defined as in Figure \ref{fig:compass}.  One property of the Brown invariant that we use below is that it is additive with respect to orthogonal direct sums of enhanced spaces.  

\begin{figure}  	
	\centering
	\includegraphics[width=8cm]{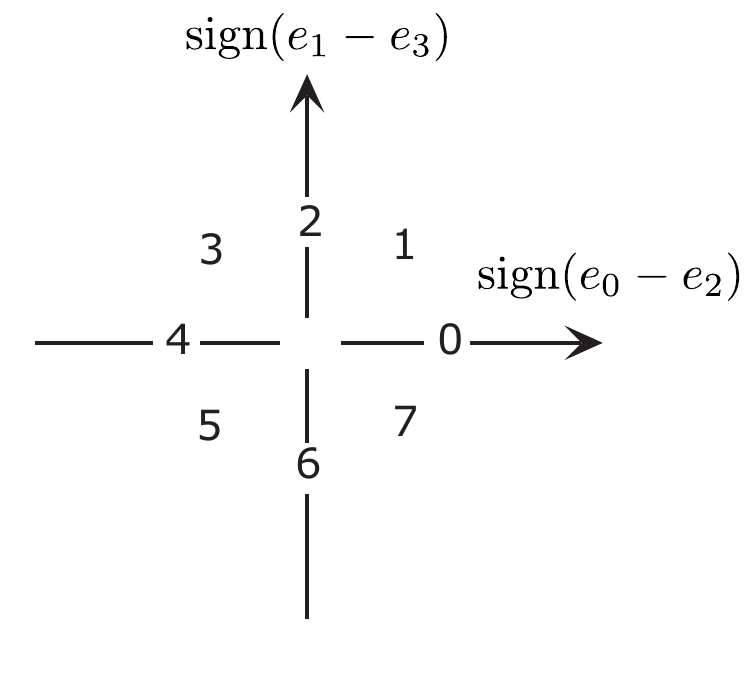}
	\caption{This figure gives the definition of the Brown invariant.  The quantities $e_i$ are the number of times that the enhancement $e$ takes on the value $i \in \mathbb{Z}/4\mathbb{Z}$ and $\text{sign}(e_0-e_2)$ and $\text{sign}(e_1-e_3)$ denote the signs (positive, zero, or negative) of the respective integers.  The quadrant is then determined by these two signs and the result is the Brown invariant.  For example, if $e_0 = e_2$ and $e_1 > e_3$, then the Brown invariant is $2$.}  
	\label{fig:compass}
\end{figure}

\begin{remark}
We could try and take this process one step further and look at $\mathbb{Z}/8\mathbb{Z}$-enhancements.  However, we find that every such enhancement $f$ is equal to $2e$ for some $\mathbb{Z}/4\mathbb{Z}$-enhancement $e$.  Similarly, suppose that 
$f: V \to \mathbb{Z}/2^n\mathbb{Z}$ is such that 
$$
	f(x+y) = f(x) + f(y) + 2^{n-1}(x \cdot y)
$$
	for all $x,y \in V$.  If $f(x)$ is odd, then since
	$$
	0 = f(x + x) = 2f(x) + 2^{n-1}(x\cdot x) \pmod{2^n}
	$$
	we find that $x \cdot x = 1$ and $n \in \{1,2\}$.  A general setup for quadratic refinements that contains what we have discussed here is considered by Taylor in \cite{taylor2001gauss} where the appearance of the number 8 is elucidated following an argument of Connolly.
\end{remark}

We now turn our attention to 3-dimensional knot theory.  Let $L$ be an unoriented link in an oriented integral homology sphere $\Sigma$.  We now define the Brown invariant of $L$, denoted $\beta(L)$  which when $L$ is not proper is $\infty$, and when $L$ is proper is valued in $\mathbb{Z}/8\mathbb{Z}$.  Note that, in contrast to $\Arf$, where we have an unoriented ambient manifold and an oriented link for $\beta(L)$, we require exactly the opposite.  If we orient $L$, then we can consider $\Arf(L)$ as in the previous section.  This is related to $\beta(L)$ by
\begin{equation}\label{eq:arf_relation}
	\beta(L) = 4 \Arf(L) + \lk(L) \pmod{8}
\end{equation}
where $\lk(L)$ is the sum of all linking numbers of all pairs of components in $L$.  This then explains the dependence of $\Arf(L)$ on the orientations of the individual components.  This also shows that $\beta(L) \in \{ 0,2,4,6\} \subset \mathbb{Z}/8\mathbb{Z}$.

To define $\beta(L)$, start by choosing an embedded (not necessarily orientable) surface $S$ in $\Sigma$ with boundary $L$.  We define a $\mathbb{Z}/4\mathbb{Z}$-valued quadratic enhancement $e_S : H_1(S; \mathbb{Z}/2\mathbb{Z}) \to \mathbb{Z}/4\mathbb{Z}$ of the intersection form on $H_1(S; \mathbb{Z}/2\mathbb{Z})$. By a \emph{band} in $\Sigma$, we mean an embedded annulus or M\"obius band in $\Sigma$.  We have a map
$$
h : \{ \text{immersed collections of bands in $\Sigma$} \} \to \mathbb{Z}/4\mathbb{Z}
$$
given (connected) bands by compatibly orienting the core of the band and the boundary, and then taking the linking number of the core with the boundary; for disconnected bands, $h$ is defined by summing over the connected components.  We then define $e_S(x)$ for $x \in H_1(F; \mathbb{Z}/2\mathbb{Z})$ by first taking an immersed (not necessarily connected) curve $C$ representing $x$ and defining 
\begin{align*}
	e_S : H_1(S; \mathbb{Z}/2\mathbb{Z}) &\to \mathbb{Z}/4\mathbb{Z} \\
	x &\mapsto h(B') + 2 d(C)
\end{align*}
where $B$ is a regular neighborhood of $C$ in $S$, $B'$ is a resolution of $B$ as in Figure \ref{fig:resolution}, and $d(C)$ is the number of double points in $C$.  The proof that this is well defined consists of showing it does not depend on the choice of resolution of $B$ to $B'$ as in Figure \ref{fig:resolution} and showing that the choice of $C$ representing $x$ does not matter.  For the former, the necessary calculation is given by computing linking numbers on of the two possible resolutions on the right side of Figure \ref{fig:resolution} in the case where $\Sigma = S^3$.  In the case of a general homology sphere, this follows from the 4-dimensional method of computing linking numbers in $\Sigma \times I$ applied to the regular homotopy between  the two possible resolutions on the right side of Figure \ref{fig:resolution} given by pushing the band on on the top right portion of Figure \ref{fig:resolution} through through the other band in order to obtain the bottom right part of Figure \ref{fig:resolution}.  The proof that the choice of $C$ representing $x$ does not matter is identical to the proof given in the case where $\Sigma = S^3$ as given in \cite{kirby_melvin}.  Note that if we switch the orientation on $\Sigma$, then the Brown invariant of $L$ changes sign since every element of $H_1(S;\mathbb{Z}/2\mathbb{Z})$ can be represented by an embedded curve and the term $h(B')$ will change signs when the orientation of $\Sigma$ is switched.  

We now check that $e_S$ is indeed a $\mathbb{Z}/4\mathbb{Z}$-valued quadratic enhancement of the intersection form on $H_1(S; \mathbb{Z}/2\mathbb{Z})$.  Let $x,y \in H_1(S; \mathbb{Z}/2\mathbb{Z})$ and let $C_x$ and $C_y$ be immersed curves in $S$ that represent $x$ and $y$, respectively.  Let $B'(C_x)$ and $B'(C_y)$ be resolutions of small regular neighborhoods of $C_x$ and $C_y$. Then $C_x \cup C_y$ represents $x+y$ with
$$
d(C_x \cup C_y) = d(C_x) + d(C_y) + x \cdot y
$$
and
$$
h(B'(C_x) \cup B'(C_y)) = h(B'(C_x)) + h(B'(C_y))
$$
thus,
\begin{align*}
	e_S(x+y) &= h(B'(C_x) \cup B'(C_y)) + d(C_x \cup C_y) \\
		 &= h(B'(C_x)) + h(B'(C_y)) + 2( d(C_x) + d(C_y) + x \cdot y)\\
		 &= e_S(x) + e_S(y) + 2(x \cdot y)
\end{align*}
We have therefore verified that $(H_1(S;\mathbb{Z}/2\mathbb{Z}),\cdot, e_S)$ is a $\mathbb{Z}/4\mathbb{Z}$-enhanced quadratic space.

\begin{figure}  	
	\centering
	\includegraphics[width=10cm]{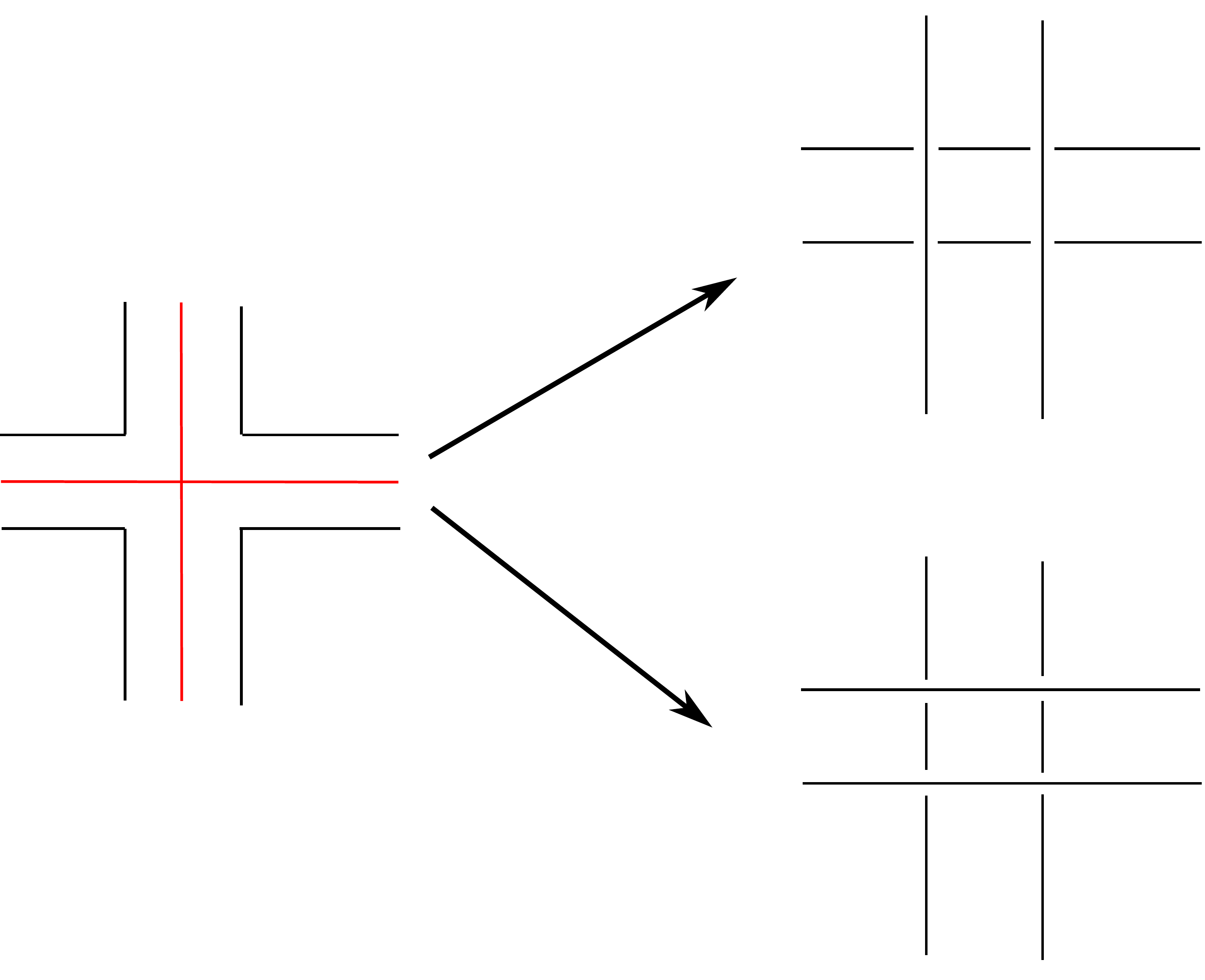}
	\caption{On the left we see a curve $C$ (in red) and a regular neighborhood $B$ of $C$ in $S$, where here, we are just showing the local picture around a point of intersection of $C$.  On the right we see the possible resolutions $B'$ of $B$ at this point of intersection.  After making one of these two choices at every point of self-intersection of $C$, we obtain a resolution $B'$ of $B$. } 
	\label{fig:resolution}
\end{figure}

Note that if $K$ is a component of $L$, then we have
$$
\lk(K, K^+) = \lk(K, L-K) \pmod 2
$$
where $K^+$ is a push-off of $K$ using the surface framing since $S$ shows that $[K^+] = [L-K]$ in $H_1(S;\mathbb{Z}/2\mathbb{Z})$.  Therefore,
\begin{align*}
	e_S([K]) &= 2 \lk(K,K^+) \\
		  &= 2 \lk(K, L-K)
\end{align*}
and thus $e_S$ is proper if and only if $L$ is proper.

We define the \emph{Brown invariant of $S$} to be the Brown invariant of $(H_1(S;\mathbb{Z}/2\mathbb{Z}),\cdot, e_S)$ and we denote this by $\beta(S)$.  We define the \emph{Brown invariant of $L$}, denoted $\beta(L)$, when $L$ is proper to be
$$
\beta(L) = \beta(S) - \phi(S)
$$
where $\phi(S)$ is half of the sum of the framings of the boundary components induced by $S$, and $\infty$ when $L$ is not proper.  

We now show that the Brown invariant for links in a homology sphere is well defined.  One way of showing that $\Arf$ is well defined is to recall that all Seifert matrices of a given link are $S$-equivalent, and then to see that the value of $\Arf$ does not change when we compute it using two different Seifert surfaces that differ by the addition of a tube as in Figure \ref{fig:add_tube} (although in the case of Seifert surfaces, we must be sure to add the tube in a way such that the resulting surface is orientable).  We mimic this proof, but we need to consider nonorientable surfaces.  Here, the relevant result is due to Gordon and Livingston \cite{mc1978signature}, which says that any two surfaces $S_1$ and $S_2$ that bound a given link $L$ in a homology sphere $\Sigma$ are related by a sequence of isotopies together with the moves shown in Figure \ref{fig:add_tube} and Figure \ref{fig:band_attach}.  

\begin{figure}  	
	\centering
	\includegraphics[width=10cm]{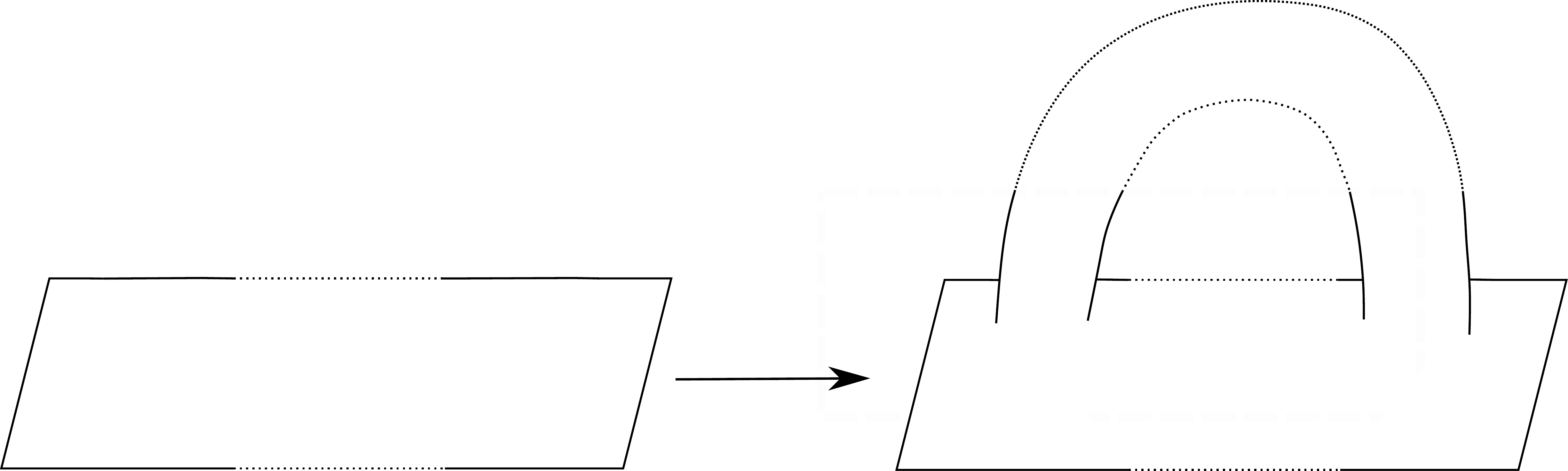}
	\caption{This figure shows one of the two moves that we will consider for changing the bounding surface for a link while not altering the bounding link.  The left side of the picture indicates the surface before the move.  The right hand side denotes the result of using this move to add a tube to the bounding surface.  It is not assumed that the two ends of the tube are in the same components of the surface and it is not assumed that the tube is added in an orientation-preserving way (the surface might not even be orientable).  Note that the bounding link is not altered.  }
	\label{fig:add_tube}
\end{figure}

\begin{figure}  	
	\centering
	\includegraphics[width=8cm]{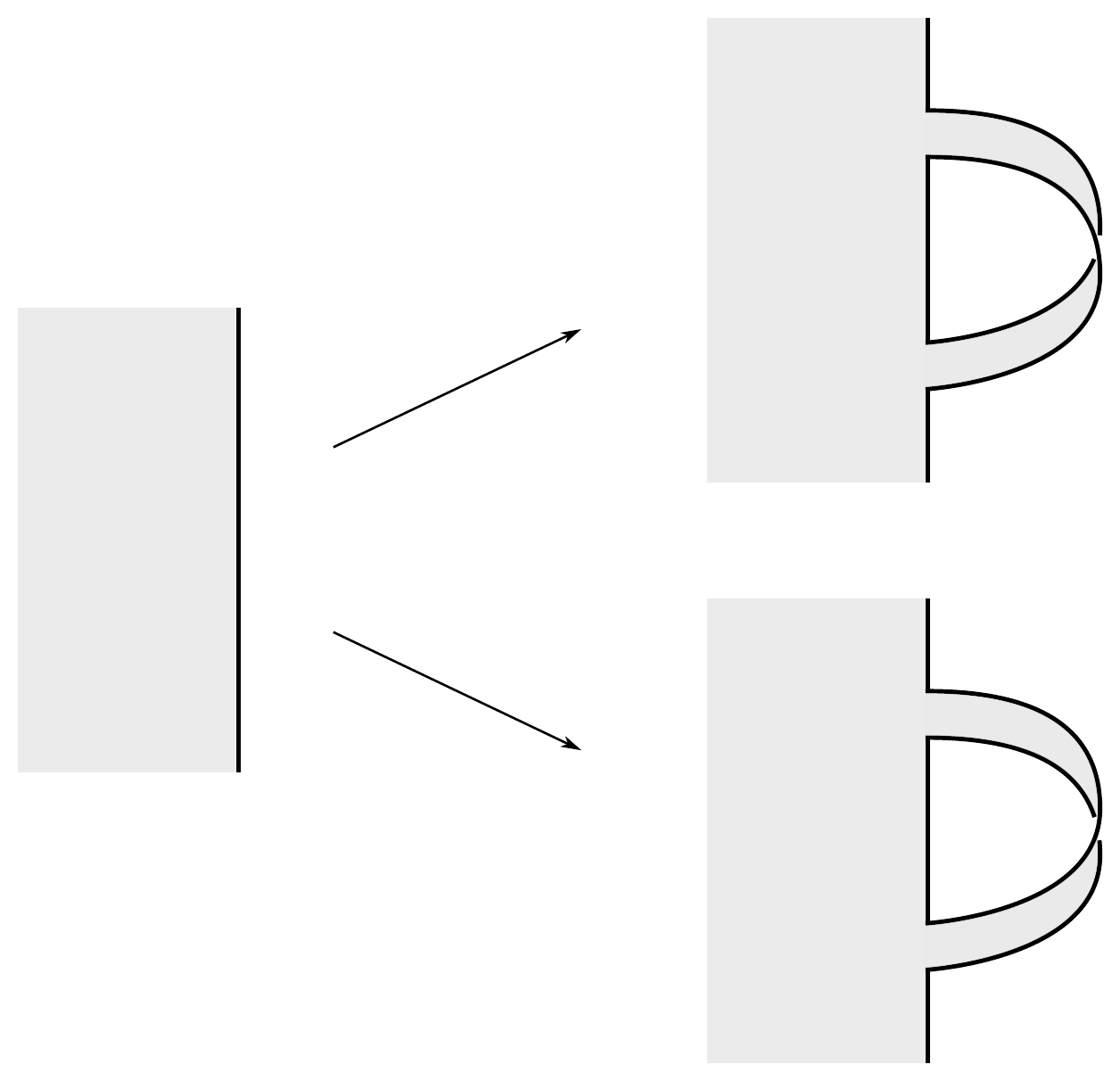}
	\caption{This figure shows one of the two moves (which has two variations) that we will consider for changing the bounding surface for a link while not altering the bounding link.  The left side of the picture indicates the surface before the move.  The the right hand side denotes the two possible modifications using this move.  Note that the bounding link is not altered.  }
	\label{fig:band_attach}
\end{figure}

To show invariance of the Brown invariant under these modifications to the bounding surface, we use an algebraic property of $\beta$, namely, that it is additive in the sense where if $(V,\cdot,e)$ is a $\mathbb{Z}/4\mathbb{Z}$-valued enhanced space and there is an orthogonal direct sum decomposition $V = V_1 \oplus V_2$, then 
$$
\beta(V, \cdot,e) = \beta(V_1,\cdot,e_{V_1}) + \beta(V_2,\cdot,e_{V_2})
$$
This follows from the Gauss sum method of computing $\beta$ as explained, for example, in \cite{kirby_melvin}.  

Let $L$ be a proper link in an oriented homology sphere $\Sigma$.  Let $S_1$ and $S_2$ be two surfaces in $\Sigma$ that bound $L$ such that $S_1$ and $S_2$ differ by the move in Figure \ref{fig:tube} where $S_2$ is the result of adding a tube to $S_1$ and where both of the ends of the tube are attached to the same connected component of $S_1$,  Then $H_1(S_2;\mathbb{Z}/2\mathbb{Z})$ splits orthogonally as 
$$
H_1(S_2;\mathbb{Z}/2\mathbb{Z}) = V \oplus (\mathbb{Z}/2\mathbb{Z}) a \oplus  (\mathbb{Z}/2\mathbb{Z}) b
$$ 
where $\beta(V) = \beta(S_1)$ as $V$ is the part of $S_2$ coming form $S_1$ and the curves $a$ and $b$ are as in Figure \ref{fig:tube}.  Then $e_{S_2}(a) = 0$ and 
$$
e_{S_2}(a+b) = e_{S_2}(a) + e_{S_2}(b) + 2(a \cdot b)
$$

Therefore, whatever $e_{S_2}(b)$ is, since we have $\beta((\mathbb{Z}/2\mathbb{Z}) a \oplus  (\mathbb{Z}/2\mathbb{Z}) b) = 0$.  Since $\phi(S_1) = \phi(S_2)$, it follows from additivity of the Brown invariant that
$$
\beta(S_1) - \phi(S_1) = \beta(S_2) - \phi(S_2)	
$$
as desired.  If instead we consider the case where $S_2$ is the result of attaching a tube to $S_1$ where the two ends of the tube are attached to different connected components, then there is no change in the homology and the proof of well definedness follows.  

\begin{figure}  	
	\centering
	\includegraphics[width=8cm]{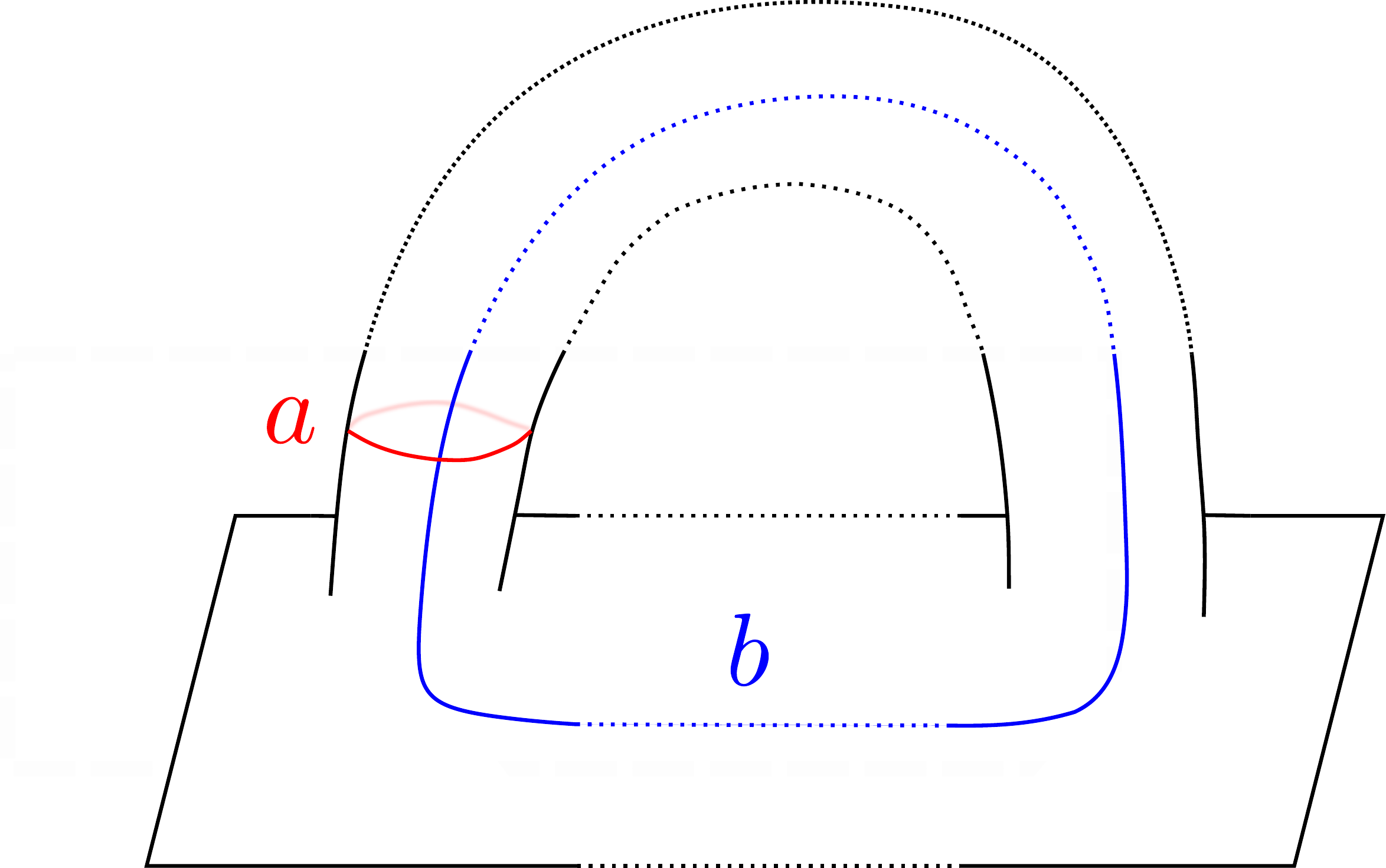}
	\caption{This figure shows the result of performing the move shown in Figure \ref{fig:add_tube} in the case where the two ends of the tube are on the same component of the surface.  We have labeled to curves $a$ and $b$, although $b$ involves making a choice.}
	\label{fig:tube}
\end{figure}

Suppose instead that $S_1$ and $S_2$ are two surfaces in $\Sigma$ that bound $L$ such that they differ by a move in Figure \ref{fig:nonorientable} (the argument for the other possibility in Figure \ref{fig:band_attach} is analogous).  Letting $a$ be the curve in Figure \ref{fig:nonorientable}, we find that $H_1(S_2;\mathbb{Z}/2\mathbb{Z})$ splits orthogonally as 
$$
H_1(S_2;\mathbb{Z}/2\mathbb{Z}) = W \oplus (\mathbb{Z}/2\mathbb{Z}) a
$$ 
where $\beta(W) = \beta(S_1)$ ($W$ is the part of $S_2$ coming form $S_1$) and since $e_{S_2}(a) = 1$, we have $\beta((\mathbb{Z}/2\mathbb{Z}) a) = 1$.  Therefore, by additivity, 
$$
\beta(S_2) = \beta(S_1) + 1
$$
The sum of the framings also changes with
$$
\phi(S_2) = \phi(S_1) + 1
$$
and therefore
$$
\beta(S_1) - \phi(S_1) = \beta(S_2) - \phi(S_2)
$$
We have thus verified that $\beta(L)$ is well defined.  

\begin{figure}  	
	\centering
	\includegraphics[width=6cm]{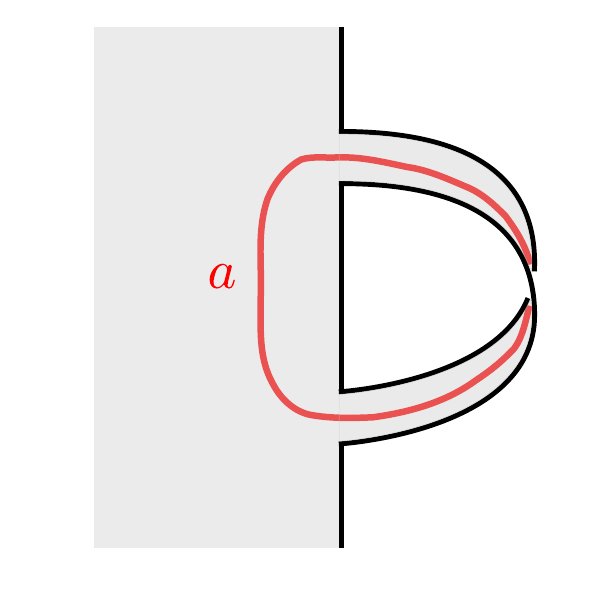}
	\caption{This figure shows the result of performing one of the two possible moves shown in Figure \ref{fig:band_attach} that alter a spanning surface for a knot without changing the knot.  We have labeled a curve $a$ on the resulting surface.  }
	\label{fig:nonorientable}
\end{figure}

Note that if $S$ is an orientable surface, then for any $x \in H_1(S; \mathbb{Z}/2\mathbb{Z})$, we can represent $x$ by a curve $C$ embedded in $S$ that has an annular regular neighborhood with boundary curves that we denote by $C_1$ and $C_2$.  Then
\begin{align*}
	e_s(x) &= \lk(C, C_1) + \lk(C, C_2) \\
	       &= 2 \lk(C, C_1) \\
	       &= 2 \lk(C, C^+) \\
	       &= 2 q_S(x)
\end{align*}
where $C^+$ is the push-off of $C$ from $S$ used in defining $q_S$.  Thus, $\beta(S) = 2 \Arf(\partial S)$. Further, letting $L_1,...,L_n$ be the components of $L$, and $L_i^+$ be the push-off of $L_i$ using the surface framing, note that 
\begin{align*}
	\phi(F) &= \frac{1}{2} \left(\lk(L_1, L_1^+) + \cdots + \lk(L_n, L_n^+) \right) \\
	        &= \frac{1}{2} \left( \sum_{i,j} \lk(L_i, L_j) \right) \\
		&= \lk(L)
\end{align*}
Therefore, equation \ref{eq:arf_relation} is verified.

Now we discuss the application of the Brown invariant to characteristic surfaces in 4-manifolds.  Note that our surfaces will not necessarily be orientable (and when they are, we recover results from the previous sections), but our 4-manifolds are always orientable and from now on oriented.  We first review the case of closed surfaces in closed 4-manifolds as in \cite{matsumoto} and \cite{guillou_manin} and then discuss the relative case that is of interest to us.  From now on, by a \emph{characteristic surface}, we mean an embedded (not necessarily orientable) surface $F^2$ in an orientable 4-manifold $X^4$ so that $F \cdot G = 0$ for all $G \in H_1(X; \mathbb{Z}/2\mathbb{Z})$, or equivalently, that $F$ is dual to $w_2(X)$.  We define a $\mathbb{Z}/4\mathbb{Z}$-valued enhancement of the intersection form $e_F : H_1(F; \mathbb{Z}/2\mathbb{Z}) \to \mathbb{Z}/4\mathbb{Z}$ as follows.  Given $x \in H_1(F; \mathbb{Z}/2\mathbb{Z})$, let $C \in F$ be a curve that represents $x$ and let $D$ be a connected orientable surface in $X$ that is immersed in $X$ with boundary $C$.  As before, the normal bundle of $D$ is trivial and we get a unique trivialization of the normal bundle of $D$ restricted to  $\partial D$.  The normal bundle of $C$ in $F$ is a line bundle contained in the normal bundle of $D$ restricted to $\partial D$ and we let $n(D)$ be the number of right-handed half-twists of this subbundle using the following convention.  Choose an orientation of $C$ and let $e_1$ and $e_2$ be sections of the normal bundle of $D$ restricted to $\partial D$ that form a basis at every point, such that, if $r_D$ is the radial outward direction of $D$ along every point of $C$, the ordered tuple of vectors $(r_D, C, e_1, e_2)$ agrees with the ambient orientation of $X$.  This convention allows us to pick out $e_1$ and $e_2$ and thus define right-handed half-twists.  Note the need to consider half-twists since $F$ is not necessarily oriented.  We then define
$$
e_F(x) = n(D) + 2(D \cdot F) + 2 d(C) \pmod 4
$$
where $d(C)$ is the number of double points of $C$ and $D \cdot F$ is the number of intersection points between $D$ and $F$.   We call $n(D)$ above the \emph{framing} of $D$.  In \cite{matsumoto} it is shown that this is well defined and it is verified that $e_F$ is a $\mathbb{Z}/4\mathbb{Z}$-valued refinement to the intersection form.  

Then the Brown invariant of $F$, denoted $\beta(F)$, is defined to be the Brown invariant of the $\mathbb{Z}/4\mathbb{Z}$-valued enhanced.  If $F$ is in fact orientable, then $n(D) = 2\mathcal{O}(D)$ and therefore $e_F(x) = 2 q_F(x)$ for all $x \in H_1(F; \mathbb{Z}/2\mathbb{Z})$; thus in this case $\beta(F) = 4 \Arf(F)$.

Given a closed (not necessarily orientable) surface $F$ in an oriented 4-manifold $X$, the \emph{self-intersection number} $F \cdot F$ is defined by pushing $F$ in the normal direction via an isotopy and counting the number of intersections of this push-off $F'$ with $F$, where the sign of a self-intersection is determined by first fixing a local orientation of $F$ at that point, giving $F'$ the same local orientation, and then seeing if the induced orientation on $X$ at the intersection point given by first using the local orientation of $F$ and then the local orientation of $F'$ agrees with $X$ or not -- these are counted as $+1$ and $-1$, respectively.  

The following theorem generalizes Theorem \ref{matsumoto} to the case of potentially nonorientable characteristic surfaces.  It is due to Guillou and Manin \cite{guillou_manin} and we are following the exposition by Matsumoto \cite{matsumoto}.  

\begin{theorem}(Guillou and Manin) \label{gm}
	Let $X^4$ be a closed oriented smooth 4-manifold with $H_1(X; \mathbb{Z}) = 0$, and let $F^2$ be a characteristic surface in $X$.  Then
	$$
	\sigma(X) = F \cdot F + 2\beta(F) \pmod{16}
	$$
\end{theorem}

Just as in the case of the Arf invariant, the condition that $H_1(X; \mathbb{Z}) = 0$ is removed by defining $\beta(F)$ for $F$ in a general closed oriented smooth 4-manifold to be the result of first doing surgery on a link disjoint from $F$ so that the resulting manifold $X'$ has $H_1(X'; \mathbb{Z}) = 0$ and then defining $\beta(F)$ as the Brown invariant of $F$ in $X'$.  By Theorem \ref{gm}, this is independent of the choice of surgery and hence well defined.   

In contrast to $\Arf$, which only depends on the homology class, the Brown invariant is not an invariant of the homology class. One such example of this is given by computing the Brown invariants of the two copies of $\mathbb{R}P^2$ contained in $S^4$ given by taking the right-handed (respectively left-handed) unknotted M\"obius bands in the present time $\mathbb{R}^3$-slice of $S^4$ and capping it off in the future time slices with an unknotted disk.

We need to be able to consider $\beta(F)$ where $F$ is not closed, but rather properly embedded in $X$.  The situation is analogous to the case of Lemma \ref{proper_relative_lemma}.  As before, the facts that we need are not exactly found in \cite{matsumoto}, so we state the results here, however they follow exactly the corresponding proofs of the statements in \cite{matsumoto} and so we omit the proofs.

\begin{lemma} \label{brown_relative_lemma}
	Let $X^4$ be a compact smooth 4-manifold with $H_1(X; \mathbb{Z}) = 0$ and let $F^2$ be a potentially nonorientable characteristic surface that is properly embedded in $X$.  Let $x \in H_1(F; \mathbb{Z}/2\mathbb{Z})$ and suppose that $x$ is represented by a curve $C \in F$ and $D^2$ is an immersed orientable surface in $X$ that is transverse to $F$.  Then define 
	\begin{align*}
		e_F : H_1(F;\mathbb{Z}/2\mathbb{Z}) &\to \mathbb{Z}/4\mathbb{Z} \\
		x &\mapsto  n(D)  + 2(D \cdot F) + 2d(C) 
	\end{align*}
	where $n(D)$ is the framing of $D$, and $d(C)$ is the number of points of $C$, and $D \cdot F$ is the number of points of intersection between $D$ and $F$.  Then $q_F$ is well defined and, together with the intersection form, makes $H_1(F; \mathbb{Z}/2\mathbb{Z})$ an enhanced space.  If $F_1$ and $F_2$ are identical characteristic surfaces, except that $F_2$ differs from $F_1$ by concordances in collar neighborhoods of the boundary components of $X$, then $\beta(F_1) = \beta(F_2)$.  If $\partial F$ is contained in an integer homology spheres boundary component of $X$, then $e_F$ is proper if and only if $\partial F$ is a proper link.  
\end{lemma}

In mimicking the proof of Theorem \ref{main_thm}, we will need a result analogous to Lemma \ref{3d-4d}.  

\begin{lemma} \label{3d-4d-brown}
Let $\Sigma$ be a integer homology sphere and let $F$ be an properly embedded surface in $\Sigma \times I$ with $L = \partial F$ a proper link.  Then 
	$$
	2\beta(L) = 2\beta(F) + F \cdot F \pmod 8
	$$
where $F \cdot F$ is computed by extending the framing on $\partial F$ that is given by the 0-framing on each of the components of $\partial F$.  
\end{lemma}

\begin{proof}

We start by showing that if $F_1$ and $F_2$ are two surfaces properly embedded in $\Sigma \times I$ with $\partial F_1 = \partial F_2$, then
	\begin{equation} \label{eq-well-def}
	F_1 \cdot F_1 + 2 \beta(F_1) = F_2 \cdot F_2 + 2 \beta(F_2) \pmod{16}
	\end{equation}
Let $W^4$ be a smooth even compact oriented 4-manifold with $H_1(W;\mathbb{Z}) = 0$ and $\partial W = \Sigma$.  Consider $\Sigma$ as sitting inside of $\Sigma \times [-1,1]$ as the $0$-cross section and orient $\Sigma \times [-1,1]$ so that the induced orientation on $\Sigma \times \{ +1 \}$ agrees with the orientation on $\Sigma$.  Consider the oriented closed smooth 4-manifold, which we will call $X$, obtained by gluing $W$ to $\Sigma \times \{-1\}$ and $\overline{W}$ to $\Sigma \times \{+1\}$.  Further, consider the characteristic closed surface in $X$ obtained by taking the union of $F_1$ and $F_2$ along $L$ where $F_1$ is in $\Sigma \times [0,1]$ and $F_2$ is in $\Sigma \times [-1,0]$.  We call this surface $F$ and note that 
$$
F \cdot F = F_1 \cdot F_1 + F_2 \cdot F_2
$$
where $F_1 \cdot F_1$ and $F_2 \cdot F_2$ are both computed using the fixed all-0-framing on $L$ as in the statement of the lemma.  Further, by Novikov additivity, we have
$$
	\sigma(X) = \sigma(W) + \sigma(\overline{W}) = 0
$$
and therefore from Theorem \ref{gm}, equation \ref{eq-well-def} follows.

	It then follows that both the right and left hand sides of equation \eqref{eq-well-def} are invariants of the link $L$ that are independent of the choice of surface.  We then fix a Seifert surface $S$ for $L$ in $\Sigma$ and we let $S'$ be the result of pushing the interior of $S$ into $\Sigma \times I$.  Now note that $S' \cdot S' = 0$, and using the orientation of $L$ induced by choosing an orientation of $S$, we have
	$$
	2 \beta(L) = 8 \Arf(L)   \pmod{16}
	$$
and
	$$
	2 \beta(F) = 8 \Arf(F) \pmod{16}
	$$
	But by Lemma \ref{3d-4d} (or more specifically, the analogous statement for links) we have
$$
	\Arf(L) = \Arf(F) \pmod{2}
$$
and therefore the result follows.  
\end{proof}

We are now in position to mimic the proof of Theorem \ref{main_thm}.  We have not yet defined $\beta(F)$ when $F$ is a properly embedded characteristic surface in a 4-manifold $X$ but $H_1(X;\mathbb{Z}) \neq 0$; however, this can be done just as it was for $\Arf(F)$ once we have proven the following result in the case where $H_1(X;\mathbb{Z}) = 0$.  Further, with this definition in hand, the following result holds regardless of the vanishing of $H_1(X; \mathbb{Z})$.

\begin{theorem} \label{brown_thm}
Let $X^4$ be a compact oriented topological 4-manifold with $\partial X$ an integer homology sphere.  Let $F^2$ be a characteristic not necessarily orientable surface that is properly embedded in $X$ such that $\partial F$ is a proper link.  Then
	$$
	2 \beta(F) + 2 \beta(\partial F) =  \sigma(X) - F \cdot F + 8 \mu(\partial X) + 8 \KS(X) \pmod{16}
	$$
where $F \cdot F$ is computed by extending the framing on $\partial F$ that is given by the 0-framing on each of the components of $\partial F$.  
\end{theorem}

\begin{proof}
	We assume that $X$ is smooth since the case where $X$ is not smooth follows from the smooth case exactly as in the proof of Theorem \ref{main_thm_top}.  We also assume that $H_1(X; \mathbb{Z}) = 0$.  Once we have addressed the case where $H_1(X; \mathbb{Z}) = 0$, it then follows that the usual definition of $\beta(F)$ extends to the case where $H_1(X; \mathbb{Z}) \neq 0$ and is well defined.  Further, the theorem follows in this case as well.  We now proceed exactly as in the proof of Theorem \ref{main_thm} where Lemma \ref{3d-4d-brown} is used in place of Lemma \ref{3d-4d}.  To obtain the topological case, we need a topological version of Theorem \ref{gm} due to Kirby and Taylor (see \cite{kirby1990pin}) where $8 \KS(X)$ is added to both sides of the equation in Theorem \ref{gm}.  
\end{proof}

For an example calculation, note that the Borromean rings have Brown invariant equal to 4 and therefore cannot bound a connected planar surface in $B^4$. Therefore, the Borromean rings cannot bound any planar surface in $B^4$.

\section{Spin structures} \label{sec:spin}

In this section, we reinterpret the results from some of the previous sections in terms of spin structures.  Many of the results and insights in this section come from \cite{kirby1990pin} and \cite{kirbybook}.  The quadratic refinements and enhancements in the previous sections are algebraic avatars for spin and pin structures, respectively and in this section, we will clarify this relationship and rederive some of the results from earlier sections with this language.  In this section, we discuss spin structures and in the next section we discuss pin structures (which, for us here, always means $\Pin^-$ structures).  The situation with spin structures is simpler than with pin structures, however, in this section we make an effort to explain the spin constructions in a language that will later be generalized when we move to pin.  

We begin by discussing spin structures and in particular by utilizing the isomorphism 
$$
\Arf : \Omega_2^{\Spin} \to \mathbb{Z}/2\mathbb{Z}
$$
which we now describe.  Given a closed connected surface $F$ with a spin structure $s$, this spin structure can be identified with a unique cohomology class $\sigma_s \in H^1(F_{\SO}(F); \mathbb{Z}/2\mathbb{Z})$, where $F_{\SO}(F)$ is the associated frame bundle of $F$, such that for every fiber the restriction map $H^1(F_{\SO}(F); \mathbb{Z}/2\mathbb{Z}) \to H^1(\SO(2); \mathbb{Z}/2\mathbb{Z})$ maps $\sigma(s)$ to the nontrivial element.  Letting $U(F)$ denote the unit tangent bundle of $F$ and note that we have a canonical identification $F_{\SO}(F) \to U(F)$ from which we obtain an element $\sigma(s) \in H^1(U(F); \mathbb{Z}/2\mathbb{Z})$.  Following Johnson \cite{johnson1980spin}, we then obtain a quadratic refinement of the intersection form on $H_1(F; \mathbb{Z}/2\mathbb{Z})$ as follows.  First, we notice that to any simple closed curve $\alpha$ in $F$ by choosing an orientation on $\alpha$, we obtain a natural section of $U(F) \to F$.  Furthermore, if we change the chosen orientation of $\alpha$, then the image of the new section is isotopic to the image of the old section.  Thus with either orientation, we will call the image of the resulting section $\vec{\alpha}$.  We then have a map
\begin{align*}
	H_1(F; \mathbb{Z}/2\mathbb{Z}) &\to H_1(U(F); \mathbb{Z}/2\mathbb{Z}) \\
	a = \sum_i^m [\alpha_i] &\mapsto  \tilde{a} := \sum_i^m \vec{\alpha_i} + mz
\end{align*}
where $z$ is the homology class of the fiber of $U(F)$ and where the curves $\alpha_i$ are pairwise disjoint.  

Then, after identifying $H^1(U(F); \mathbb{Z}/2\mathbb{Z})$ with $\Hom(H_1(U(F); \mathbb{Z}/2\mathbb{Z}), \mathbb{Z}/2\mathbb{Z})$ and considering $\sigma(s)$ in the latter, we have the quadratic refinement
\begin{align*}
	q_s : H_1(F; \mathbb{Z}/2\mathbb{Z}) &\to \mathbb{Z}/2\mathbb{Z} \\
	x &\mapsto \sigma_s(\tilde{x})
\end{align*}

We then define $\Arf(F,s)$ as $\Arf(q_s)$.  Johnson showed that this map from spin structures on $F$ to quadratic refinements of the intersection form on $H_1(F; \mathbb{Z}/2\mathbb{Z})$ is equivariant with respect to the natural free and transitive actions of $H^1(F; \mathbb{Z}/2\mathbb{Z})$ on both sets and therefore is a bijection.  

We now give a different description of $q_s$ that relies on Johnson's work to show that it is well defined.  We mention this because we find it to be the most transparent description of $q_s$.  Recall that there are two inequivalent spin structures on $S^1$ and $\Omega_1^{\Spin} = \mathbb{Z}/2\mathbb{Z}$ so one of these spin structures is called the \emph{Lie group spin structure} (the spin structure that is nonzero in $\Omega_1^{\Spin}$ and is the trivial double cover of $S^1$) and the \emph{bounding spin structure} (the spin structure that is induced by taking any spin structure on a spin manifold bounding $S^1$ and taking the induced spin structure on the boundary).  Given an oriented bundle $\xi$, we let $\mathcal{S}pin(\xi)$ denote the set of equivalence classes of Spin structures on $\xi$ that respect the orientation.  Using the inclusion $\Spin(n) \to \Spin(n+1)$, there is then a bijection 
$$
\mathcal{S}pin(\xi) \to \mathcal{S}pin(\xi + \epsilon)
$$
where $\epsilon$ is an oriented line bundle.  Let $x \in H_1(F; \mathbb{Z}/2\mathbb{Z})$ and let $\alpha$ be a simple closed curve in $F$ representing $x$.  If we give $\alpha$ an orientation, then the normal bundle $\nu_{\alpha \subset F}$ obtains a natural orientation so that the  orientation of $T\alpha$ together with orientation of $\nu_{\alpha \subset F}$ agrees with the ambient orientation of $F$.  Then we can restrict the spin structure on $F$ to a spin structure on $T\alpha + \nu_{\alpha \subset F}$ and then this induces a unique spin structure on $\alpha$ which is either the Lie group or the bounding spin structure.  Additionally, the choice of orientation of $\alpha$ does not affect the resulting element of $\Omega_1^{\Spin}$.  

We now show that $q_s(x) = 0$ if the induced spin structure on $\alpha$ is the bounding spin structure and $q_s(x) = 1$ if the induced spin structure on $\alpha$ is the Lie group spin structure.  To see this, first note that by definition 
\begin{align*}
	q_s(x) &= \sigma_s(\vec{\alpha} + z) \\
	       &= \sigma_s(\vec{\alpha}) + 1
\end{align*}
so we must understand $\sigma_s(\vec{\alpha})$.  Having oriented $\alpha$, we obtain an inclusion $i_{\vec{\alpha}} : F_{\SO}(\alpha) \to F_{\SO}(F)$ and the induced spin structure on $S^1$ corresponds to the cohomology class 
$$
i_{\vec{\alpha}}^*(\sigma_s) \in H^1(F_{\SO}; \mathbb{Z}/2\mathbb{Z}) \cong H^1(\alpha; \mathbb{Z}/2\mathbb{Z}) \cong \mathbb{Z}/2\mathbb{Z}
$$
Where $0 \in \mathbb{Z}/2\mathbb{Z}$ corresponds to the Lie group spin structure and $1 \in \mathbb{Z}/2\mathbb{Z}$.  Since 
$$
i_{\vec{\alpha}}^*(\sigma_s)([\alpha]) = \sigma_s(\vec{\alpha})
$$
the result follows.  

%We mention now another descripion of $q_s$ that will be generalized in the next section to pin structures.   

\begin{remark}
We could discuss spin structures on an oriented bundle $\xi$ without mention of a choice of metric by talking about $\widetilde{GL^+}$-principal bundles that map equivariantly to $F_{\GL^+}(\xi)$ where $\widetilde{\GL^+}$ is the universal cover of $\GL^+$.  However, the tradition is to prefer working with compact Lie groups and we will follow suit.  Thus in what follows there will be choices of metrics when appropriate, although ultimately none of these choices are important.  
\end{remark}

We now return to the situation of Section \ref{sec:main}, but through the lens of spin structures.  Let $X^4$ be a closed oriented connected smooth 4-manifold and let $F^2$ be a closed connected oriented characteristic surface in $X^4$.  The manifold $X-F$ admits a spin structure that does not extend over $F$ and the group $H^1(X; \mathbb{Z}/2\mathbb{Z})$ acts freely and transitively on the set of such spin structures (see \cite{kirby1990pin}).  Following \cite{kirby1990pin}, we call such a spin structure a \emph{spin characterization} of $F$ in $X$ and we denote the set of equivalence classes of such spin characterizations by $\mathcal{S}pin\mathcal{C}har(X,F)$.  We now describe a map
$$
\mathcal{S}pin\mathcal{C}har(X^4,F^2) \to \mathcal{S}pin(F^2)
$$
(see page 66 of \cite{kirbybook}).  Let $E$ denote the normal bundle of $F$ in $X$ and let $\partial E$ denote the unit normal bundle of $F$ in $X$.  Then from the spin structure on $X - F$, $\partial E$ inherits a spin structure.  The bundle $\partial E$ admits a section after removing a point $\ast$ of $F$.  Choose such a section $sec$.  Note that we now have a splitting as the sum of two line bundles
$$
\nu_{F \subset X} = \lambda + \epsilon
$$
where the trivial bundle $\epsilon$ is trivialized from the choice of section $s$ and $\lambda$ is the orthogonal complement of $\epsilon$.  Therefore, $\lambda$ inherits a unique orientation such that the orientation of $F$ together with the orientation of $\lambda$ and $\epsilon$ agrees with the ambient orientation of $X$.  Note that we have a canonical isomorphism
$$
\nu_{sec(F - \ast) \subset \partial E} \cong \lambda
$$
from which we orient $\nu_{sec(F - \ast )i\subset \partial E}$.  Then using the spin structure on $\partial E$ together with this orientation of $\nu_{sec(F - \ast) \subset \partial E}$, we obtain a spin structure on $sec(F - \ast)$ and hence on $F$.  That this is well defined (i.e., independent of the choice of section $s$ and $\ast$) is shown in \cite{kirbybook} (see page 66) and follows from the spin structure on $X-F$ being characteristic.    

This construction works analogously in all dimensions but we only have use for the aforementioned case.  The construction also works analogously for $X$ and $F$ with boundary and $F$ properly embedded in $X$, and for $F$ disconnected where characterizations are required to not extend over any component of $F$.  It may be helpful for the reader to keep in mind the analogous construction for orientations.  Namely, given a codimension-1 embedded submanifold $V$ dual to $w_1$, there is a map from the set of orientations on the complement of $V$ to the set of orientations on $V$ (see Lemma 2.2 as well as Proposition 2.3 of \cite{kirby1990pin}).  

In the case where $H_1(X; \mathbb{Z}/2\mathbb{Z}) = 0$, then there is a unique characterization of $F$ in $X$ and thus we get a unique spin structure $s$ on $F$ by applying the above construction.  Then by Johnson's construction, this yields a unique quadratic refinement of the intersection form
$$
q_s : H_1(F; \mathbb{Z}/2\mathbb{Z}) \to \mathbb{Z}/2\mathbb{Z}
$$
In section \ref{sec:main}, we discussed a different such quadratic refinement $q_F$.  We now show that $q_s = q_F$.  Let $x \in H_1(F; \mathbb{Z}/2\mathbb{Z})$.  Let $C$ be an embedded curve in $F$ that represents $x$ and let $D$ be an embedded surface in $X$ with $\mathcal{O}(D) = 0$ (i.e., $D$ is framed, which can always be achieved by boundary twisting).  Then $q_F(x) = D \cdot F$ so it suffices to verify that $q_s(x) = D \cdot F$.  

Given an $n$-dimensional bundle $\xi$ over $S^1$ with a spin structure $s$  on $\xi$, we get a corresponding element $[\xi, s]$ of $\Omega_1^{\Spin} \cong \mathbb{Z}/2\mathbb{Z}$, corresponding to whether $(\xi, s)$ is isomorphic to the appropriate stabilization of the Lie group spin structure on $TS^1$ or the bounding spin structure on $TS^1$.  If $G$ is an orientable surface with boundary and $\xi$ is a bundle over $F$ with spin structure $s$, then
$$
\sum_{C \in \pi_0(\partial F)} [C, s_C] = 0 \pmod 2 
$$
Note that $q_s(x) = [sec(C), c|_{sec(C)}]$ where $c$ is the given characterization of $F$ in $X$ and further, we assume that the section $sec$ has been chosen so that $sec(C) \subset D$.  Let $G = D \cap (X- \interior(E))$ where $E$ is the unit normal bundle of $F$ in $X$ and note that since $c$ is a characteristic, then for every point in $D \cap F$, there is a boundary component $b$ of $G$ with $[b, c|b] = 1$.  It thus follows that
\begin{align*}
	q_s(x) &= [sec(C), c|_{sec(C)}] \\
	       &= \sum_{b \in \pi_0(G), b \neq sec(C)} [b, c|_{b}] \\
	       &= D \cdot F
\end{align*}
as desired.  

In \cite{kirbybook}, an equivalent form of Theorem \cite{matsumoto} is proven in the language of spin structures and follows from a calculation of a certain bordism.  Namely, the  following is shown:

\begin{theorem}(Kirby) \label{kirby_rochlin}
Let $X^4$ be a closed oriented smooth 4-manifold, let $F$ be an oriented characteristic surface in $X$, and let $c$ be a charicterization of $F$ in $X$.  Then
$$
	\Arf(F, s(c)) = \frac{\sigma(X) - [F]^2}{8} \pmod 2	
$$
where $s(c)$ is the spin structure on $F$ induced by the characterization $c$.  
\end{theorem}

It follows from this result that, even when $H_1(X; \mathbb{Z}) \neq 0$, $\Arf(F)$ as in section \ref{sec:main} agrees with $\Arf(F, s(c))$ for any choice of characterization $c$ of $F$ in $X$.

The following result relating the Arf invariant of a knot in a homology sphere to the Arf invariant of a certain spin surface is also proved in \cite{kirbybook} (see page 69).

\begin{lemma} \label{rob:arf}
	Let $\Sigma$ be a integer homology sphere and let $F$ be a connected orientable properly embedded surface in $\Sigma \times I$ with connected boundary.  There is a unique spin structure on $\Sigma \times I - F$ that does not extend over $F$ and this induces a unique spin structure on $F$.  Let $\overline{F}$ denote the result of attaching a disk to the $\partial F$ and extending the spin structure on $F$ to this disk, and let $s$ denote the resulting spin structure on $\overline{F}$.  Then $\Arf( \partial F) = \Arf(\overline{F}, s)$.  
\end{lemma}

We are now in position to give an alternative proof of Theorem \ref{main_thm} using spin structures.

\begin{proof} (Alternative proof of Theorem \ref{main_thm})
We need one fundamental additional additivity observation.  Namely, if $F_1$ and $F_2$ are two compact surfaces with spin structures, then, by removing a disk from both $F_1$ and $F_2$ and identifying the boundaries with an orientation-reversing diffeomorphism, the connect sum $F_1 \sharp F_2$ inherits a natural spin structure and 
$$
\Arf(F_1 \sharp F_2) = \Arf(F_1) + \Arf(F_2)
$$
This can be seen by noting that $F_1 \sharp F_2$ is spin cobordant to the disjoint union of $F_1$ and $F_2$ or by using the quadratic enhancement description of spin structures on surfaces together with the fact that the additivity of the Arf invariant.  

	Note that, since $H^1(\partial X; \mathbb{Z}) = 0$ and $\partial F$ is dual to $w_2$ in $\partial X$, there is a unique spin structure on $\partial X - \partial F$ that does not extend over $\partial F$, which we call the \emph{characterization} of $\partial F$ in $\partial X$.   We now proceed exactly as in the previous proof of Theorem \ref{main_thm}, where we use Theorem \ref{kirby_rochlin} in place of Theorem \ref{matsumoto}, and we adopt the notation without further mention.  Let $c$ denote a characterization of $F$ in $X$ and note that restricted to the boundary this results in the characterization $\partial c$ of $\partial F$ in $\partial X$.  Note that $W - F_W$ also admits a (unique) characterization of $F_W$ in $W$ and that when restricted to the boundary, this induces the same characterization $\partial c$ of $\partial F$ in $\partial X$, although with the orientation reversed.  Therefore,	 these two characterizations of $F$ in $W$ and of $F_W$ in $W$ glue together to yield a characterization of $\overline{F}$ in $X \cup W$.  Note that the induced spin structure on $F \cup F_W$  coming from this characterization, which we call $s$, when restricted to $F$ or $F_W$, agrees with the spin structure on $F$ or $F_W$ obtained from the characterization of $F$ in $X$ or of $F_W$ in $W$, respectively.  

	It follows from Lemma \ref{rob:arf} together with our previous comments on why the spin definition of Arf agrees with the previous definition for a characteristic surface, we have 
	\begin{align*}
		\Arf(F \cup F_W, s) &= \Arf(F, s|_F) + \Arf(F_W, s|_{F_W}) \\
				    &= \Arf(F) + \Arf(\partial F) 
	\end{align*}
From this, together with the same calculation as in the previous proof of Theorem \ref{main_thm}, the result follows.  

\end{proof}

We end this section by discussing the case where $\partial F$ is a link as in section \ref{sec:links} and we follow the notation from that section.  From the perspective of spin structures, we can start with a characterization $c$ of $X -F$ and this will induce a spin structure on $F$ which therefore gives a spin structure on each of the components of $\partial F$.  These will all be the bounding spin structures (regardless of the choice of characterization $c$) if and only if $\partial F$ is a proper link in $\partial X$.  The reason is that the induced form $q_{s(c)}$ from the induced spin structure $s(c)$ on $F$ agrees with the form from section \ref{sec:links}, thus the discussion in section \ref{sec:links} applies.  Then by using a version of Lemma \ref{rob:arf} adjusted to proper links, we obtain a proof of relative version of Rochlin's theorem for links mentioned in section \ref{sec:links} in the language of spin structures.

\section{Pin structures} \label{sec:pin}

By a pin structure, we always mean a $\Pin^-$ structure and we closely follow \cite{kirby1990pin}, which we recommend for background information.  The discussion in this section is terser than the last section as much of the material here is considered in great detail in \cite{kirby1990pin}.  The appearance of $\Pin^-$ as opposed to $\Pin^+$ in the discussion that follows ultimately stems from the fact that all compact surfaces admit $\Pin^-$ structures as well as the more interesting low-dimensional bordism groups that $\Pin^-$ has, namely:
\begin{align*}
	\Omega_1^{\Pin^-} &= \mathbb{Z}/2\mathbb{Z}   &\Omega_1^{\Pin^+} &= 0 \\
	\Omega_2^{\Pin^-} &= \mathbb{Z}/8\mathbb{Z}   &\Omega_1^{\Pin^+} &= \mathbb{Z}/2\mathbb{Z}
\end{align*}
What is particularly relevant for us here is the isomorphism
$$
\beta : \Omega_2^{\Pin^-} \to \mathbb{Z}/8\mathbb{Z}
$$
which is given by taking a compact surface with pin structure $(F, p)$, associating a quadratic enhancement to the intersection form 
$$
e_p : H_1(F; \mathbb{Z}/2\mathbb{Z}) \to \mathbb{Z}/4\mathbb{Z}
$$
and then taking the Brown invariant
$$
\beta(F,p) := \beta(H_1(F;\mathbb{Z}/2\mathbb{Z}), \cdot, e_p)
$$
The definition of $e_p$ together with the fact that the map $\beta$ is an isomorphism are discussed in section 3 of \cite{kirby1990pin}.  In the case where we consider a spin structure $s$ on $F$ as a pin structure using the inclusion $\Pin^-(2) \to \Spin(2)$, we have
$$
e_s = 2 q_s
$$
and therefore
$$
\beta(F, s) = 4 \Arf(F,s)
$$
Let $\Pin^-(F)$ denote the set of equivalence classes of pin structures on $F$.  

We now discuss the relevant descend-of-structure results from section 6 of \cite{kirby1990pin}.  Let $X^4$ be a closed oriented connected smooth 4-manifold and let $F^2$ be a compact characteristic surface properly embedded in $X^4$.  The manifold $X-F$ admits a pin structure that does not extend over any component of $F$ and the group $H^1(X; \mathbb{Z}/2\mathbb{Z})$ acts freely and transitively on the set of such pin structures (see \cite{kirby1990pin}).  Following \cite{kirby1990pin}, we call such a pin structure a \emph{pin characterization} of $F$ in $X$ and we denote the set of equivalence classes of such pin characterizations by $\mathcal{P}in^-\mathcal{C}har(X,F)$.  There is a map
\begin{equation} \label{eq:pin_descend}
\mathcal{P}in^-\mathcal{C}har(X^4,F^2) \to \mathcal{P}in^-(F^2)
\end{equation}
which is natural with regards to the respective cohomology group actions and, when $F$ is characteristic, fits into a commuting square with the corresponding map for spin structures described in the previous section.  

The set of pin characterizations of $F$ in $X$ is acted on freely and transitively by $H^1(X; \mathbb{Z}/2\mathbb{Z})$ and so, in the case where $H_1(X; \mathbb{Z}/2\mathbb{Z})$, there is a unique pin characterization of $F$ in $X$.  In this case, using the map in (\ref{eq:pin_descend}), there is an induced pin structure $p$ on $F$ which corresponds to a quadratic enhancement
$$
e_p : H_1(F; \mathbb{Z}/2\mathbb{Z}) \to \mathbb{Z}/4\mathbb{Z}
$$
In section \ref{sec:brown}, a different such quadratic enhancement $e_F$ was described, and in fact 
$$
e_p = e_F
$$
(see the discussion on page 221 of \cite{kirby1990pin}).  Note that this gives another proof of the analogous result for spin structures from section \ref{sec:spin}.   

In the case where $X$ and $F$ have nonenmpty boundary, we call a characterization $c$ of $F$ in $X$ \emph{even} if the induced pin structure on $F$ given by (\ref{eq:pin_descend}), which we denote by $p(c)$, has the property that on every boundary component, it yields the bounding pin structure.  In this case we define the Brown invariant of $F$ in $X$ to be 
$$
\beta(F, c) := \beta(\overline{F}, p(c))
$$
where $\overline{F}$ is the result of capping off all of the boundary components of $F$ with disks and extending the pin structure $p(c)$ to a pin structure on all of $\overline{F}$, which we again denote by $p(c)$.

We now discuss how to obtain the invariant $\beta(L)$ from section \ref{sec:brown} in the language of pin structures.  We first consider a more general situation following \cite{kirby1990pin}, which we recommend for additional details. Let $M^3$ be an arbitrary  3-manifold with a spin structure $s$ and let $L$ be a link in $M$ with $[L] = 0 \in H_1(M; \mathbb{Z}/2)$ (and thus $[L]$ is dual to $w_2(M)$).  A \emph{pin characterization} of $L$ in $M$ is, as before, a pin structure on $M-L$ that does not extend across any component of $L$.  As before, we have an analogously defined descend of structure map
$$
\mathcal{P}in^-\mathcal{C}har(M^3,L^1) \to \mathcal{P}in^-(L^1)
$$
and we call a characterization $c$ of $L$ in $M$ \emph{even} if the induced pin structure on each component of $L$ is the bounding pin structure.  

Given $(M, L, s, c)$ we obtain a class $\gamma \in H^1(M - L; \mathbb{Z}/2\mathbb{Z})$ which is the unique cohomology class that acts on $c$ to obtain $s|_{M-L}$.  Let $E$ be the total space of the open disk bundle of $L$ in $M$ and let $S$ be the total space of the circle bundle of $E$.  A \emph{set of longitudes} for $L$ is a choice of parallel push-off of each component of $L$ in $M$.  We will call such a set of longitudes $l$ \emph{even} in $(M,s,c)$ if the dual of the class $\gamma$ (thought of in the cohomology of the complement of unit disk bundles of $L$) is represented by a surface $F^2$ that intersects $S$ in $l$.  If $l$ is an even set of longitudes and $l'$ is another set of longitudes, then $l'$ is even if and only if each component of $l'$ differs from the corresponding component of $l$ by an even number of twists. 

There is then an invariant 
$$
\beta(L,s,c,l) \in \mathbb{Z}/8\mathbb{Z}
$$
which is obtained as follows.  There exists an embedded surface $S^2$ in $M$ with $\partial S = L$ so that $S$ is dual to the element $\gamma$ and so that the longitudes of $L$ induced by $S$ are $l$.  The surface $F$ then obtains a pin structure $p(s)$ from the spin structure $s$ as is explained on pages 233-234 of $\cite{kirby1990pin}$ and this pin structure has the property that, when restricted to the boundary components, it yields the bounding pin structure.  We then take
$$
\beta(L,s,c,l) := \beta(\overline{S}, p(s))
$$
where $\overline{S}$ is the result of capping off all of the boundary components of $S$ with disks and extending the pin structure $p(s)$ to a pin structure on all of $\overline{S}$, which we again denote by $p(s)$.

Let $F^2$ be a surface properly embedded in a 4-manifold and let $l$ be a set of longitudes for $\partial F$.  Let $F \cdot_l F$ denote the self-intersection of $F$ where we use a push-off of $F$ that extends the push-off along $\partial F$ given by $l$.  In this notation, the self-intersection considered in the previous relative-Rochlin-type theorems were always $F \cdot_0 F$, where $0$ denotes the 0-framing.   The following is analogous to Lemma \ref{3d-4d-brown} and will be used to go between 3-dimensional and 4-dimensional invariants.  

\begin{lemma} \label{lemma:final3d4d}
	Let $M^3$ be a compact 3-manifold with spin structure $s$.  Let $L$ be a link in $M$ with $[L] = 0 \in H_1(M; \mathbb{Z}/2\mathbb{Z})$ and let $c$ be an even characterization of $L$ in $M$ and let $l$ be a choice of even longitudes for $L$ in $M$.  Let $F$ be a compact connected characteristic surface in $M \times I$ where we denote the boundary components of $M \times I$ by $M_0$ and $M_1$ and where $\partial F \subset M_0$. There is a unique characterization $C$ of $F$ in $L$ such that $C|_{M_0} = c$ and using the induced pin structure on $F$ from $C$ has the property that all of the boundary components have the bounding pin structure.    Then
$$
	2\beta(L, s, c, l) = 2\beta(F, C) + F \cdot_l F \pmod{16}
$$
\end{lemma}

\begin{proof}

Using the free and transitive action of $H^1(M \times I; \mathbb{Z}/2\mathbb{Z})$ on the set of characterizations of $F$ in $M \times I$, we see that there is a unique such characterization $C$ with $C|_{M_0} = c$, from which it immediately follows that the induced pin structure on $F$ from $C$ has the property that all of the boundary components have the bounding pin structure.

	Isotope $F$ so that $F \cap M \times [0,1/4] = L \times [0,1/4]$ and let $F' = F \cap M \times [1/4, 1]$.  Then the normal bundle of $F'$ in $M \times [1/4, 1]$ admits a section and let $l'$ denote the even longitudes of $\partial F' \subset M_{1/4} := M \times \{ 1/4 \}$.  The longitudes $l'$ on $L$ differ from the longitudes $l$ on $L$ by a total of $2r$ right handed twists (summing over the components of $L$) for some integer $r$.  By Theorem 8.3 of \cite{kirby1990pin}, we have 
	$$
	\beta(L, s, c, l') = \beta(L, s, c, l) + r \pmod{8}
	$$
	By Theorem 8.2 of \cite{kirby1990pin} applied to $M \times [1/4, 1]$ and $F'$ (note the missing factors of $2$ in front of all of the $\beta$ terms), we have 
$$
	\beta(L, s,c,l') = \beta(F', C') \pmod{8}
$$
	where here $C'$ is $C$ restricted to $M \times [1/4, 1]$.  Note that $2r = F \cdot_l F$ where all of the self intersections can be assumed to take place within $M \times [0,1/4]$ due to changing of the framings.  The result then follows.  
\end{proof}

\begin{lemma} \label{lemma:pinnified}
Let $\Sigma^3$ be an integral homology sphere and let $L$ be link in $\Sigma$.  Let $s$ is the unique spin structure on $\Sigma$ and let $c$ is the unique characterization of $L$ in $\Sigma$.  Then $L$ is a characteristic link if and only if $c$ is even.  Let $0$ denotes the 0 push-off of each component of $L$, then $0$ is an even set of longitudes for $L$ with respect to $s$ and $c$.  If $L$ is characteristic, then
$$
	\beta(L) = \beta(L, s, c, 0) \pmod 8
$$
	where the invariant on the left is the Brown invariant of $L$ as defined in section \ref{sec:brown}. 

\end{lemma}

\begin{proof}
	For $L$ characteristic, it follows from Lemma \ref{lemma:final3d4d} that
$$
	2\beta(L, s, c, 0) = 2\beta(F, C) + F \cdot_0 F \pmod{16}
$$
	and from Lemma \ref{3d-4d-brown} we have 
$$
	2\beta(L) = 2\beta(F) + F \cdot_0 F \pmod{16}
$$ 
	From the proceeding discussion of how our pin structure framework generalizes the results in section \ref{sec:brown}, we know that 
	$$
	\beta(F) = \beta(F, C) \pmod{8}
	$$
and thus the result follows.  
\end{proof}

Following the alternative proof of Theorem \ref{main_thm} given in the previous section using spin structures, Theorem \ref{brown_thm} can be derived for $X^4$ smooth in the language of pin structures by using Lemma \ref{lemma:pinnified} to identify $\beta(L)$ with an invariant related to pin structures together with Lemma \ref{lemma:final3d4d}. 

We need the following topological version of Theorem \ref{gm} due to Kirby and Taylor \cite{kirby1990pin} (which we already used in the proof of Theorem \ref{brown_thm}):

\begin{theorem}  \label{final_closed} (Kirby and Taylor)
	Let $X^4$ be a closed oriented  4-manifold and let $F^2$ be a characteristic surface in $X$.  Let $C$ be a pin characterization of $F$ in $X$ .  Then
	$$
	2 \cdot \beta(F,C) = F \cdot F - \sigma(X) + 8 \KS(X) \pmod{16}
	$$
\end{theorem}

Let $M^3$ be a compact 3-manifold with a choice of spin structure $s$.  Then the Rochlin invariant of $(M,s)$, denoted $\mu(M,s)$, is defined as 
$$
\mu(M,s) := \sigma(W) \pmod{16} 
$$
where $W^4$ is a smooth compact 4-manifold with a spin structure that extends the given spin structure on $M$.  By Rochlin's theorem, this is well defined.  

The following theorem generalizes all of the previously given theorems.  The main results in the previous sections could have all been proved for $F$ and $\partial X$ not necessarily connected, however, for simplicity, we chose not to do this.  In the following theorem, we do not make these restrictions.  This result yields a ``combinatorial'' formula for $\KS(X^4)$ for $X$ orientable, compact, and with nonempty boundary (see Theorem \ref{thm:closed_top} for the closed case).  The proof strategy is the same as usual: cap off the boundary components appropriately, apply the relevant formula for closed manifolds, and then use various additivity results together with a lemma (in this case Lemma \ref{lemma:final3d4d}) to relate a 4-dimensional invariant to a 3-dimensional invariant. 

\begin{theorem} \label{thm:pin}
	Let $X^4$ be a compact oriented topological 4-manifold.  Let $F^2$ be a characteristic not necessarily orientable surface that is properly embedded in $X$ such that $\partial F$ is a proper link  representing the 0 element in $H_1(M; \mathbb{Z}/2\mathbb{Z})$ in each of the components $M^3$ of $\partial X$.  For each boundary component $M$ of  $\partial X$, let $s_M$ be a choice of spin structure on $M$.  Let $C$ be a choice of even characterization of $F$ in $X$ and let $l$ be a choice of even longitudes for $F$ with respect to the spin structures $s_M$ and $C|_{M}$.   Then
	$$
	2 [\beta(F,C) + \sum_{M} \beta(\partial F \cap M, s_M, C|_{M}, l \cap M)] = \sigma(X) - F \cdot_l F + 8 \KS(X) + \sum_{M} \mu(M, s_M)  \pmod{16}
	$$
where the sums are over the boundary components $M$ of $\partial X$.  
\end{theorem}

\begin{proof}
	For each boundary component $M$ of $X$, let $W_M$ be a compact smooth 4-manifold with $\partial W_M = M$ and suppose that $W_M$ has a spin structure that bounds the inverse of the spin structure on $M$.  Let $F_M$ be a compact embedded surface in a collar boundary of $W_M$ obtained from isotoping a surface bounding $F \cap M$ in $M$ into $W_M$ (and thus $F_M$ is dual to $w_2(W_M)$) with $\partial F_M = \partial F$.  Then by definition
	\begin{align} \label{eq:final1}
	\sigma(W_M) = \mu(M, -s_M) = - \mu(M, s_M) \pmod{16}
	\end{align}
The characterization $C|_{M}$ extends to a characterization of $F_M$ in $W_M$ using Lemma \ref{lemma:final3d4d} together with the fact that $W_M$ is spin.  Let $C_M$ denote a choice of such an extension.

	Let $\overline{X}$ be closed manifold resulting from gluing all of the manifolds $W_M$ to $X$ and let $\overline{F}$ be the result of closed surface that is the union of $F$ and all of the surfaces $F_M$. Let $\overline{C}$ denote the characterization of $\overline{F}$ in $\overline{X}$ that is the union of the characterization $C$ and the characterizations $C_M$.  Then we have
	\begin{align*} 		
		2\beta(\overline{F}, \overline{C}) + \overline{F} \cdot \overline{F} &= 2\beta(F,C) + \sum_{M} 2\beta(F_M, C_M) + F \cdot_l F + \sum_{M} F_M \cdot_{l \cap M} F_M \\
										     &= 2 [\beta(F,C) + \sum_{M} \beta(\partial F \cap M, s_M, C|_{M}, l \cap M)] + F \cdot_l F \pmod{16} \numberthis \label{eq:final2}
	\end{align*}
	where the first equality follows by computing $F \cdot F$ using pushoffs extending the push-off given by $l$ together with the additivity of the Brown invariant and the fact that $\partial F$ is characteristic, and the second equality follows from Lemma \ref{lemma:final3d4d}.   
	Then by applying Theorem \ref{final_closed} to $\overline{X}$ and $\overline{F}$, together with Novikov additivity, equation \eqref{eq:final1}, and equation \eqref{eq:final2}, we obtain the desired result. 
\end{proof}

\bibliography{arf}
\bibliographystyle{alpha}

\end{document}